\documentclass[final,onefignum,onetabnum]{siamart220329}

\usepackage{amsmath,amsfonts,amsopn,amssymb,subeqnarray}
\usepackage{graphicx}
\graphicspath{ {Figures/} }
\usepackage[caption=false]{subfig}
\usepackage{multirow,relsize,makecell}
\usepackage{url}
\usepackage{enumitem}
\ifpdf

\DeclareGraphicsExtensions{.eps,.pdf,.png,.jpg}
\else
  \DeclareGraphicsExtensions{.eps}
\fi
\usepackage{algorithm, algorithmic}

\usepackage{esint}

\numberwithin{theorem}{section}
\newsiamremark{assumption}{Assumption}
\newsiamremark{remark}{Remark}
\newsiamremark{example}{Example}
\crefname{assumption}{Assumption}{Assumptions}
\crefname{remark}{Remark}{Remarks}
\crefname{example}{Example}{Examples}

\headers{Additive Schwarz for the \MakeLowercase{$p$}-Laplacian}{Young-Ju~Lee and Jongho~Park}

\title{On the linear convergence of additive Schwarz methods for the \MakeLowercase{$p$}-Laplacian\thanks{Submitted to arXiv.
\funding{Young-Ju Lee's work was supported by NSF-DMS 2208499, Shapiro Fellowship from Penn State University in the Spring of 2022 and REP grant for the year of 2022, from Texas State University. 
Jongho Park’s work was supported by the
National Research Foundation of Korea~(NRF) grant funded by the Korean government~(MSIT) (No. 2021R1C1C2095193).}
}}

\author{
Young-Ju Lee\thanks{Department of Mathematics, Texas State University, San Marcos, TX 78666, USA
  (\email{yjlee@txstate.edu}).}
\and
Jongho Park\thanks{Applied Mathematics and Computational Sciences Program, Computer, Electrical and Mathematical Science and Engineering Division, King Abdullah University of Science and Technology~(KAUST), Thuwal 23955, Saudi Arabia  (\email{jongho.park@kaist.ac.kr}, \url{https://sites.google.com/view/jonghopark}).}
}

\ifpdf
\hypersetup{
  pdftitle={On the linear convergence of additive Schwarz methods for the p-Laplacian},
  pdfauthor={Young-Ju Lee and Jongho Park}
}
\fi

\newcommand\cT{\mathcal{T}}
\newcommand\tw{\tilde{w}}
\newcommand\up{\underline{p}}
\newcommand\op{\overline{p}}
\newcommand\hp{\hat{p}}

\newcommand\PF{\mathrm{PF}}
\newcommand\QM{\mathrm{QM}}
\newcommand\SD{\mathrm{SD}}

\DeclareMathOperator*{\argmin}{\arg\min}

\begin{document}

\maketitle

\begin{abstract}
We consider additive Schwarz methods for boundary value problems involving the $p$-Laplacian.
While existing theoretical estimates suggest a sublinear convergence rate for these methods, empirical evidence from numerical experiments demonstrates a linear convergence rate.
In this paper, we narrow the gap between these theoretical and empirical results by presenting a novel convergence analysis.
Firstly, we present a new convergence theory for additive Schwarz methods written in terms of a quasi-norm.
This quasi-norm exhibits behavior akin to the Bregman distance of the convex energy functional associated with the problem.
Secondly, we provide a quasi-norm version of the Poincar'{e}--Friedrichs inequality, which plays a crucial role in deriving a quasi-norm stable decomposition for a two-level domain decomposition setting.
By utilizing these key elements, we establish the asymptotic linear convergence of additive Schwarz methods for the $p$-Laplacian.
\end{abstract}



\section{Introduction}
\label{Sec:Introduction}
Let $\Omega$ be a bounded polygonal domain in $\mathbb{R}^2$ with the Lipschitz boundary $\partial \Omega$.
Given $p \in (1,\infty)$, we consider the following $p$-Laplace equation: 
\begin{equation}
\begin{split}
\label{pLap_strong}
- \nabla \cdot \left( |\nabla u |^{p-2} \nabla u \right) = f \quad &\text{ in } \Omega, \\
u = 0 \quad &\text{ on } \partial \Omega,
\end{split}
\end{equation}
where $f \in L^{p^*}(\Omega)$ with $p^*$ being from the equation $1/p + 1/p^* = 1$.

The $p$-Laplacian is a standard example of nonlinear elliptic problems~\cite{BGKT:2018}. Furthermore, it has a number of application areas, including glaciology, non-Newtonian fluids~\cite{Shapovalov:2017}, nonlinear diffusion, and nonlinear elasticity; see~\cite{Diaz:1985} and references therein.
Thus, there has been extensive research on~\eqref{pLap_strong}, especially for numerical solutions of~\eqref{pLap_strong}.
Some important early results can be found in~\cite{Ciarlet:2002,GM:1975}.
Finite element methods for the $p$-Laplacian were analyzed in terms of the quasi-norm in~\cite{BL:1993,BL:1994}.
Further studies on error estimates for the $p$-Laplacian in terms of the quasi-norm were conducted in~\cite{CLY:2006,EL:2005,LY:2001,LY:2002}.
Linear convergence of adaptive finite element methods for~\eqref{pLap_strong} was shown in~\cite{DK:2008}.
Numerical homogenization for multiscale $p$-Laplacian problems was investigated in~\cite{LCZ:2021}.

It is well-known that the boundary value problem~\eqref{pLap_strong} can be formulated in the following weak form~\cite{Ciarlet:2002,GM:1975}: find $u \in W_0^{1,p}(\Omega)$ such that 
\begin{equation*} 
\int_{\Omega} |\nabla u|^{p-2} \nabla u \cdot \nabla v \, dx = \int_{\Omega} f \, v \,dx, \quad v \in W^{1,p}_0(\Omega),  
\end{equation*}
where $W_0^{1,p} (\Omega)$ is a usual Sobolev space consisting of the $L^p (\Omega)$-functions vanishing on $\partial \Omega$ with $(L^p (\Omega))^2$-gradient.
Equivalently, it is interpreted as the following convex optimization problem:
\begin{equation}\label{pLap}
\min_{v \in W_0^{1,p} (\Omega)} \left\{ F(v): = \frac{1}{p} \int_{\Omega} |\nabla v|^p \,dx - \int_{\Omega} f v \,dx \right\}. 
\end{equation}
That is, one may deal with the convex optimization problem~\eqref{pLap} to obtain a solution of~\eqref{pLap_strong}.
Based on the convex optimization formulation~\eqref{pLap}, multigrid and preconditioned descent methods were proposed in~\cite{BI:2000} and~\cite{HLL:2007}, respectively. In particular, the framework of subspace correction methods~\cite{Xu:1992} for~\eqref{pLap} were considered in~\cite{Park:2020,TX:2002}.

This paper is concerned with numerical solutions of boundary value problems involving the $p$-Laplacian by additive Schwarz methods.
Additive Schwarz methods, also known as parallel subspace correction methods, have been broadly used as efficient numerical solvers for large-scale scientific problems; see~\cite{TW:2005,Xu:1992} and references therein for relevant results on linear problems.
In additive Schwarz methods, the domain of a target problem is decomposed into a union of several subdomains, and optimal local corrections on the subdomains with respect a numerical approximation for the solution are computed in parallel.
The numerical approximation for the solution is iteratively updated by collecting all the local corrections.
Due to their parallel structures, additive Schwarz methods are suitable for massively parallel computation using distributed memory computers.
In the past decades, there have been a number of results on additive Schwarz methods for large-scale convex optimization problems.
The framework of additive Schwarz methods was first considered for convex optimization in~\cite{TE:1998}, and subsequently applied to the $p$-Laplacian in~\cite{TX:2002}.
These methods have since been further investigated in several studies, including~\cite{Badea:2006,Badea:2019,Park:2020,Park:2022}.

The convergence rate of additive Schwarz methods for the $p$-Laplacian problem~\eqref{pLap_strong} was first analyzed in~\cite{TX:2002}; the  $\mathcal{O} (n^{-\frac{\up(\up-1)}{(\op-\up)(\op+\up-1)}})$ energy convergence of the methods was proven, where $n$ denotes the number of iterations, $\up = \min \{p, 2 \}$, and $\op = \max \{ p, 2 \}$.
Recently,~\cite{Park:2020} showed that the methods satisfy the improved $\mathcal{O} (n^{-\frac{\op(\up-1)}{\op-\up}})$ convergence rate~(see \cref{Prop:Park}). 
The results in both~\cite{TX:2002} and~\cite{Park:2020} are based on some estimates for the Bregman distance of the energy functional $F$ in~\eqref{pLap}.
Roughly speaking, these estimates are written as
\begin{equation}
\label{Bregman_rough}
    \mu_{\op} \| u - v \|_{W^{1,p}(\Omega)}^{\op} \leq D_F (u,v) \leq L_{\up} \| u - v \|_{W^{1,p}(\Omega)}^{\up}, \quad u, v \in W_0^{1,p}(\Omega),
\end{equation}
where $\mu_{\op}$ and $L_{\up}$ are positive constants independent of $u$ and $v$, and $D_F (u,v)$ is the Bregman distance of $F$ defined by
\begin{equation}
\label{Bregman_distance}
    D_F (u,v) = F(u) - F(v) - \left< F'(v), u - v \right>, \quad u, v \in W_0^{1,p}(\Omega).
\end{equation}
Here, $F' (v)$ stands for the Frech\'{e}t derivative of $F$ at $v$ given by
\begin{equation}
\label{Frechet}
    \left< F'(v), w \right> = \int_{\Omega} |\nabla v|^{p-2} \nabla v \cdot \nabla w \,dx - \int_{\Omega} fw \,dx,
    \quad w \in W_0^{1,p} (\Omega).
\end{equation}
One may refer to~\cite[Lemma~2.1]{TX:2002} and~\cite[Section~6.1]{Park:2020} for details on the estimate~\eqref{Bregman_rough}.

While both~\cite{TX:2002} and~\cite{Park:2020} proved the sublinear convergence of additive Schwarz methods for the $p$-Laplacian, it was observed numerically in several works that the methods actually converge linearly; see, e.g.,~\cite[Fig.~2]{Park:2021}.
Indeed, as we will demonstrate in the numerical experiments presented in \cref{Sec:Numerical} of this paper, additive Schwarz methods for~\eqref{pLap_strong} exhibit linear convergence empirically under various settings on discretization and domain decomposition.
More precisely, each convergence curve of the energy error with respect to the number of iterations seems linear in the $x$-linear $y$-log scale plot when the number of iterations is sufficiently large, which means that the energy error decays exponentially as the number of iterations increases.
This implies that the existing convergence estimates for additive Schwarz methods for the $p$-Laplacian may not be optimal.

The main motivation of this paper is to discuss a linear convergence analysis for additive Schwarz methods to solve the $p$-Laplacian problem~\eqref{pLap_strong}.
As we mentioned above, while the existing theoretical estimates~\cite{Park:2020,TX:2002} for the convergence rate of additive Schwarz methods for the $p$-Laplacian are sublinear, the empirical convergence rate observed by numerical experiments is linear.
This discrepancy between theoretical and empirical results motivates our work, as we aim to bridge this gap by rigorously proving the asymptotic linear convergence of additive Schwarz methods for the $p$-Laplacian.

In~\eqref{Bregman_rough}, $\up$ and $\op$ do not agree if $p \neq 2$, so that the lower and upper bounds for $D_F (u,v)$ are expressed in powers of $\| u - v \|_{W^{1,p}(\Omega)}$ with different exponents.
This discrepancy indicates that a power of norm is not adequate as a tight two-sided approximation for the Bregman distance; whenever we establish a bound for $D_F (u,v)$ in terms of $\| u - v \|_{W^{1,p}(\Omega)}$ or vice versa, we suffer from a kind of looseness.
We claim that the sublinear convergence rates given in the existing works~\cite{Park:2020,TX:2002} are caused by this looseness.
To overcome this issue, we propose to use the quasi-norm developed in~\cite{CLY:2006,EL:2005,LY:2001,LY:2002}, which is relevant to the problem of consideration and approximates the Bregman distance appropriately, and then to derive the convergence estimate in terms of the quasi-norm.
This approach is similar to obtain the convergence measure of the iterative method using the energy-like metric relevant to the problem to be solved, as discussed in~\cite{LWXZ:2008,LWC:2009}.
We denote the quasi-norm by $\| \cdot \|_{(\nabla v)}$~(see~\eqref{quasi_norm}) and show that
\begin{equation}
    \label{Bregman}
    \mu_p \| u-v \|_{(\nabla v)}^2 \leq D_F (u,v) \leq L_p \| u - v \|_{(\nabla v)}^2, \quad u, v \in W_0^{1,p}(\Omega)
\end{equation}
for some positive constants $\mu_p$ and $L_p$~(see \cref{Lem:Bregman}), i.e., $\| u - v \|_{(\nabla v)}^2$ approximates $D_F (u,v)$ well up to a multiplicative constant.
Meanwhile, we note that the quasi-norm $\| \cdot \|_{(\nabla v)}$, along with several alternative versions described in~\cite{DK:2008,DR:2007}, do not induce a norm.
As a result, existing convergence theories for additive Schwarz methods~\cite{Park:2020,TX:2002} cannot directly utilize the estimate~\eqref{Bregman}.
A novelty in this paper is that, by extending the idea of ~\cite{Park:2020}, a new convergence theory for additive Schwarz methods is obtained in terms of the quasi-norm, which utilizes~\eqref{Bregman} to obtain the asymptotic linear convergence rate of additive Schwarz methods for the $p$-Laplacian.
In our linear convergence analysis, a quasi-norm version of the Poincar\'{e}--Friedrichs inequality~(see \cref{Lem:Poincare_greater,Lem:Poincare_less}) plays a critical role.
We validate this asymptotic linear convergence result numerically in \cref{Sec:Numerical}.

The rest of this paper is organized as follows.
In \cref{Sec:ASM}, we present finite element approximations, domain decomposition settings, and a two-level additive Schwarz method for the $p$-Laplacian problem.
An asymptotic linear convergence analysis of the two-level additive Schwarz method is given in \cref{Sec:Convergence}.
In \cref{Sec:Poincare}, we present details of the quasi-norm Poincar\'{e}--Friedrichs inequality that is used in the convergence analysis of the methods.
In \cref{Sec:Numerical}, we provide numerical results of the two-level additive Schwarz method for the $p$-Laplacian problem across various settings.
Finally, we provide a concluding remark for our paper in \cref{Sec:Conclusion}.

\section{Additive Schwarz methods}
\label{Sec:ASM}
In this section, we introduce finite element spaces and domain decomposition settings for the $p$-Laplacian problem~\eqref{pLap}.
Based on these settings, we present a two-level additive Schwarz method for~\eqref{pLap} and its convergence theory, which explains the asymptotic linear convergence of the algorithm.

In what follows, the notation $A \lesssim B$ means that there exists a constant $c > 0$ such that $A \leq c B$, where $c$ is independent of the geometric parameters $H$, $h$, and $\delta$ relying on discretization and domain decomposition.
We also write $A \approx B$ if $A \lesssim B$ and $B \lesssim A$.

\subsection{Discretization and domain decomposition}
Let $\cT_h$ be a quasi-uniform triangulation of $\Omega$ with $h$ the characteristic element diameter.
The collection of continuous and piecewise linear functions on $\cT_h$ vanishing on $\partial \Omega$ is denoted by $V = S_h (\Omega)$.
Clearly, we have $V \subset W_0^{1, \infty} (\Omega)$.
For continuous functions, the nodal interpolation operator $I_h$ onto $S_h (\Omega)$ is well-defined.

In what follows, we consider the following conforming finite element approximation of~\eqref{pLap} defined on $V$:
\begin{equation}
    \label{pLap_FEM}
    \min_{u \in V} F(u).
\end{equation}
A unique solution of~\eqref{pLap_FEM} is denoted by $u^* \in V$.
Convergence properties of~\eqref{pLap_FEM} as $h \rightarrow 0$ can be found in~\cite{BL:1993,Ciarlet:2002}.

Next, we describe domain decomposition settings for the problem~\eqref{pLap_FEM}.
We assume that $\Omega$ admits another quasi-uniform triangulation $\cT_H$ with $H$ the characteristic element diameter such that $\cT_h$ is a refinement of $\cT_H$.
A finite element space $S_H (\Omega)$ is defined in the same manner as $S_h (\Omega)$.
In the two-level additive Schwarz method for~\eqref{pLap_FEM}, $\cT_h$ and $\cT_H$ will play roles of fine and coarse meshes, respectively.
Let $\{ \Omega_k \}_{k=1}^N$ be a nonoverlapping domain decomposition of $\Omega$ such that each $\Omega_k$ is the union of several coarse elements in $\cT_H$ and the number of coarse elements consisting of $\Omega_k$ is uniformly bounded.
For each subdomain $\Omega_k$, $1 \leq k \leq N$, we consider an enlarged region $\Omega_k'$ consisting of the elements $T \in \cT_h$ with $\operatorname{dist} (T, \Omega_k) \leq \delta$.
Then $\{ \Omega_k'\}_{k=1}^N$ forms an overlapping domain decomposition of $\Omega$.
We define $S_h (\Omega_k ') \subset W_0^{1, \infty}(\Omega_k')$ as the piecewise linear finite element space on the $\cT_h|_{\Omega_k'}$ with the homogeneous essential boundary condition.

We set
\begin{equation*}
V_0 = S_H (\Omega), \quad
V_k = S_h (\Omega_k'), \quad 1 \leq k \leq N.
\end{equation*}
A two-level domain decomposition for $V$ is given by
\begin{equation}
\label{2L}
V = \sum_{k=0}^N R_k^* V_k,
\end{equation}
where $R_k^*$:~$V_k \rightarrow V$, $1 \leq k \leq N$, is the natural extension-by-zero operator and $R_0^*$:~$V_0 \rightarrow V$ is the natural interpolation operator.
Let $\{ \theta_k \}_{k=1}^N$ be the piecewise linear partition of unity for $\Omega$ subordinate to the covering $\{ \Omega_k' \}_{k=1}^N$ that was presented in~\cite[Eq.~(3.7)]{TW:2005}.
It is known that $\{ \theta_k \}_{k=1}^N$ satisfies the following properties:
\begin{subequations}
\label{pou}
\begin{align}
\theta_k = 0 \quad \text{ on } \Omega \setminus \Omega_k',& \label{pou1} \\
\sum_{k=1}^N \theta_k = 1 \quad \text{ on } \overline{\Omega}, \label{pou2}& \\
\| \nabla \theta_k \|_{L^{\infty} (\Omega_k')} \lesssim \frac{1}{ \delta}, \quad 1 \leq k \leq N.& \label{pou3}
\end{align}
\end{subequations}
The following lemma summarizes an important result on stable decomposition for the two-level domain decomposition~\eqref{2L}~(see~\cite[Lemma~4.1]{TX:2002}).

\begin{lemma}
\label{Lem:2L}
For $w \in V$, let $w_0 \in V_0$ be the $L^2 (\Omega)$-orthogonal projection of $w$ onto $V_0$ and let $w_k \in V_k$, $1 \leq k \leq N$, such that
\begin{equation*}
    R_k^* w_k = I_h (\theta_k (w - R_0^* w_0)).
\end{equation*}
For $s \geq 1$, we have $w = \sum_{k=0}^N R_k^* w_k$ and
\begin{equation*}
    \sum_{k=0}^N | R_k^* w_k |_{W^{1,s} (\Omega)}^s \lesssim \left( 1 + \left( \frac{H}{\delta} \right)^{s-1} \right) | w |_{W^{1,s} (\Omega)}^s.
\end{equation*}
\end{lemma}

Using the usual coloring technique, one can prove that the two-level domain decomposition~\eqref{2L} enjoys the strengthened convexity condition~(see~\cite[Assumption~4.2]{Park:2020}).

\begin{lemma}
\label{Lem:convex}
Let $N_c$ be the minimum number of colors such that $\{ \Omega_k' \}_{k=1}^N$ is colored in a way that the subdomains with the same color do not intersect with each other, and let $\tau_0 = 1/(N_c + 1)$.
For any $v \in V$, $w_k \in V_k$, $0 \leq k \leq N$, and $\tau \in (0, \tau_0]$, we have
\begin{equation*}
    (1 - \tau (N+1)) F(v) + \tau \sum_{k=0}^N F(v + R_k^* w_k ) \geq F \left( v + \tau \sum_{k=0}^N R_k^* w_k \right).
\end{equation*}
\end{lemma}
\begin{proof}
See~\cite[Section~5.1]{Park:2020}.
For suitable overlaps, we have $N_c = 4$~\cite{TX:2002}.
\end{proof}

\subsection{Two-level additive Schwarz method}
The two-level additive Schwarz method for~\eqref{pLap_FEM} based on the space decomposition~\eqref{2L} is described in \cref{Alg:ASM}.
It is worth noting that this algorithm has been investigated in several prior works.
The algorithm for smooth convex optimization was first considered in~\cite{TE:1998}, and then applied to the $p$-Laplacian in~\cite{TX:2002}.
Later, the framework was generalized to constrained and nonsmooth convex optimization problems in~\cite{Badea:2006} and~\cite{Park:2020,Park:2021}, respectively.
The constant $\tau_0$ in \cref{Alg:ASM} was given in \cref{Lem:convex}.

\begin{algorithm}
\caption{Two-level additive Schwarz method for~\eqref{pLap_FEM}}
\label{Alg:ASM}
\begin{algorithmic}[0]
\STATE Let $u^{(0)} \in V$ and $\tau \in (0, \tau_0]$.
\FOR{$n=0,1,2,\dots$}
\STATE \begin{equation*}
\begin{split}
w_k^{(n+1)} &= \argmin_{w_k \in V_k} F(u^{(n)} + R_k^* w_k ), \quad 0 \leq k \leq N \\
u^{(n+1)} &= u^{(n)} + \tau \sum_{k=0}^N R_k^* w_k^{(n+1)}
\end{split}
\end{equation*}
\ENDFOR
\end{algorithmic}
\end{algorithm} 

The following proposition summarizes the sublinear convergence rate of \cref{Alg:ASM} analyzed in~\cite[Theorem~6.1]{Park:2020}.
It was discussed in~\cite[Remark~4.2]{Park:2021} that the rate presented in \cref{Prop:Park} is the sharpest estimate among the existing ones~\cite{Badea:2006,BK:2012,Park:2020,TX:2002}.

\begin{proposition}
\label{Prop:Park}
In \cref{Alg:ASM}, we write $\zeta_n = F (u^{(n)}) - F(u^*)$ for $n \geq 0$.
There exist positive constants $\zeta^*$ and $c^*$, depending on $u^{(0)}$, $\tau$, and $H/\delta$, such that
\begin{equation*}
    \zeta_{n+1} \leq
    \begin{cases}
        \displaystyle \left(1 - \tau \left( 1 - \frac{1}{\up} \right) \right) \zeta_n, & \quad \zeta_n \geq \zeta^*, \\
        \displaystyle \zeta_n - c^* \zeta_n^{\frac{\up (\op - 1)}{\op (\up - 1)}}, & \quad \zeta_n < \zeta^*,
    \end{cases}
\end{equation*}
where $\up = \min \{ p, 2\}$ and $\op = \max \{p, 2 \}$.
Consequently, we have
\begin{equation*}
    \zeta_n \lesssim \frac{1}{(c^* (n+1))^{\frac{\op (\up - 1)}{\op - \up}}}
\end{equation*}
for sufficiently large $n \geq 0$.
\end{proposition}
\begin{proof}
    See~\cite[Section~A.4]{Park:2020}.
\end{proof}

While \cref{Prop:Park} ensures the sublinear convergence of \cref{Alg:ASM}, as we will see in \cref{Sec:Numerical},  the actual numerical behavior indicates linear convergence.
This observation motivates us to develop a new convergence theory for \cref{Alg:ASM} that can explain the linear convergence.
We summarize our main result, the asymptotic linear convergence of \cref{Alg:ASM}, in \cref{Thm:linear}.
The proof of \cref{Alg:ASM} will be provided in \cref{Sec:Convergence}.
We highlight that \cref{Thm:linear} stands as the first theoretical result that explains the linear convergence of the additive Schwarz method for the $p$-Laplacian.

\begin{theorem}
\label{Thm:linear}
If the solution $u^* \in V$ of~\eqref{pLap_FEM} satisfies that $| \nabla u^* |$ does not vanish on $\Omega$, then, in \cref{Alg:ASM}, we have
\begin{equation*}
\limsup_{n \rightarrow \infty} \frac{F(u^{(n+1)}) - F(u^*)}{F(u^{(n)}) - F(u^*)} \leq 1 - \gamma^{-1},
\end{equation*}
where $\gamma$ is a positive constant depending on $p$, $u^*$, $H$, $\delta$, and $\tau$ such that
\begin{equation*}
\gamma \lesssim \left[ 1 + \overline{C}_{p,u^*}^{\PF} \left( \frac{H}{\delta} \right)^p \right]^{\frac{1}{\min \{p, 2 \} - 1}},
\end{equation*}
and the constant $\overline{C}_{p,u^*}^{\PF}$ is given in either \cref{Lem:Poincare_greater} or \cref{Lem:Poincare_less}.
\end{theorem}

Regarding the condition in \cref{Thm:linear} requires a condition that the finite element solution $u^*$ should satisfy $| \nabla u^* | \neq 0$ on $\Omega$, we discuss its validity under extreme values of $p$, particularly when $p$ is either very large or close to $1$.
As we will demonstrate in \cref{Sec:Numerical}, for large $p$, the solution may develop a singularity~(see \cref{Fig:sol}(e, f)).
Fortunately, this singularity does not violate the condition $| \nabla u^* | \neq 0$.
However, when $p$ close to $1$, the solution may exhibit a flat region, potentially leading to a vanishing gradient~(see \cref{Fig:sol}(a, b)).
Consequently, the applicability of \cref{Thm:linear} to cases near $p = 1$ may be limited.

Despite the potential limitations in applying \cref{Thm:linear} to such cases, it remains practically relevant, as many real-world applications involving the $p$-Laplacian typically utilize moderate values of $p$.
For instance, in modeling nonlinear Darcy law for fluid flows, as discussed in~\cite{BGKT:2018}, physically meaningful values for $p$ are generally greater than $3/2$.

We conclude this section by mentioning several acceleration methodologies that can be applied to \cref{Alg:ASM}.
In~\cite{Park:2021,Park:2022}, acceleration schemes for additive Schwarz methods for convex optimization were proposed.
As the energy functional $F$ is convex, these schemes can be directly applied to \cref{Alg:ASM} to yield accelerated variants.
These accelerated methods show faster convergence behaviors than the vanilla method, while they have essentially the same computational cost per iteration; see~\cite{Park:2022} for relevant numerical results.
We do not deal with the accelerated methods in detail because they are beyond the scope of this paper.

\section{Convergence analysis}
\label{Sec:Convergence}
The main objective of this section is to prove \cref{Thm:linear}, which is the asymptotic linear convergence theorem for the two-level additive Schwarz method for the $p$-Laplacian.
We begin by presenting some useful properties of the quasi-norm $\| \cdot \|_{(\nabla v)}$~\cite{EL:2005,LY:2001}, which is defined as
\begin{equation}
\label{quasi_norm}
    \| w \|_{(\nabla v)}^2 = \int_{\Omega} \left( | \nabla w | + | \nabla v | \right)^{p-2} |\nabla w|^2 \,dx, \quad v, w \in W^{1,p} (\Omega).
\end{equation}
Subsequently, we prove \cref{Thm:linear} by verifying a certain quasi-norm stable decomposition property~\cite{Park:2020,TX:2002}.

\subsection{Properties of the quasi-norm}
The quasi-norm $\| \cdot \|_{(\nabla v)}$ given in~\eqref{quasi_norm} satisfies a scaling property in the sense that the $\| tw \|_{(\nabla v)}$ is bounded by $\| w \|_{(\nabla v)}$ multiplied by $t^\alpha$ for some $\alpha \in \mathbb{R}$, where $v,w \in W^{1,p} (\Omega)$ and $t \in [0,1]$.
\Cref{Lem:scaling} summarizes such a property.

\begin{lemma}
\label{Lem:scaling}
For any $v,w \in W^{1,p} (\Omega)$ and $t \in [0, 1]$, we have
\begin{equation*}
    t^{\max \{p, 2\} } \| w \|_{(\nabla v)}^2
    \leq \| tw \|_{(\nabla v)}^2
    \leq t^{\min \{p, 2\} } \| w \|_{(\nabla v)}^2.
\end{equation*}
\end{lemma}
\begin{proof}
Suppose that $p \in [2, \infty)$.
Since the map $x \mapsto x^{p-2}$~($x \geq 0$) is increasing, we get
\begin{align*}
    \| tw \|_{(\nabla v)}^2
    &\leq \int_{\Omega} \left( |\nabla w| + |\nabla v| \right)^{p-2} |t \nabla w|^2 \,dx
    = t^2 \| w \|_{(\nabla v)}^2, \\
    \| tw \|_{(\nabla v)}^2
    &\geq \int_{\Omega} \left( t |\nabla w| + t |\nabla v| \right)^{p-2} |t \nabla w|^2 \,dx
    = t^p \| w \|_{(\nabla v)}^2.
\end{align*}
The case $p \in (1, 2)$ can be shown by a similar argument using the fact that the map $x \mapsto x^{p-2}$~($x \geq 0$) is decreasing.
\end{proof}
 
The following lemma states that $\| u - v \|_{(\nabla v)}$ is bounded by $\| u - v \|_{(\nabla u)}$ up to a multiplicative constant independent of $u,v \in W^{1,p} (\Omega)$.

\begin{lemma}
\label{Lem:symmetry}
For any $u,v \in W^{1,p} (\Omega)$, we have
\begin{equation*}
    \| u - v \|_{(\nabla v)}^2 \leq 2^{|p-2|} \| u - v \|_{(\nabla u)}^2.
\end{equation*}
\end{lemma}
\begin{proof}
Invoking the vector inequality
\begin{equation*}
    |\xi + \eta| + |\xi| \leq 2 \left( |\xi + \eta| + |\eta| \right), \quad \xi, \eta \in \mathbb{R}^2,
\end{equation*}
we get
\begin{multline*}
    \| u - v \|_{(\nabla v)}^2 = \int_{\Omega} \left( |\nabla (u-v)| + |\nabla v| \right)^{p-2} |\nabla (u-v)|^2 \,dx \\
    \leq 2^{|p-2|} \int_{\Omega} \left( |\nabla (v-u)| + |\nabla u| \right)^{p-2} |\nabla (v-u)|^2 \,dx = 2^{|p-2|} \| u - v \|_{(\nabla u)}^2,
\end{multline*}
which completes the proof.
\end{proof}

In~\cite{BL:1993,BL:1994}, the following vector inequalities were established: there exist two positive constants $C_1$ and $C_2$ such that, for any $\xi, \eta \in \mathbb{R}^2$, the following hold:
\begin{subequations}
\label{BL}
\begin{align}
\label{BL_upper}
\left| |\xi|^{p-2} \xi - |\eta|^{p-2} \eta \right| &\leq C_1 |\xi - \eta| \left( |\xi| + |\eta| \right)^{p-2},
\\
\label{BL_lower}
\left(|\xi|^{p-2} \xi - |\eta|^{p-2} \eta \right) \cdot \left ( \xi - \eta \right) &\geq C_2 |\xi - \eta|^{2} \left( |\xi| + |\eta| \right)^{p-2}.
\end{align}
\end{subequations}
Using~\eqref{BL} and proceeding similarly to~\cite[Theorem~2.1]{BL:1993}, we prove \cref{Lem:Bregman}, which says that the estimate~\eqref{Bregman} actually holds.
\Cref{Lem:Bregman} will play an important role in proving~\eqref{Lem:stable}; see also~\cite[Lemma~2.3]{LCZ:2021}.

\begin{lemma}
\label{Lem:Bregman}
There exists positive constants $\mu_p$ and $L_p$ depending on $p$ such that, for any $u,v \in W^{1,p} (\Omega)$, we have
\begin{equation*} 
\mu_p \| u - v \|_{(\nabla v)}^2 \leq D_F (u,v) \leq L_p \| u - v \|_{(\nabla v)}^2. 
\end{equation*} 
\end{lemma}
\begin{proof}
By the definition of $D_F(u,v)$ given in~\eqref{Bregman_distance} and the fundamental theorem of calculus, we have
\begin{equation*} \begin{split} 
D_F (u,v) &= \int_0^1 \langle F'(v + t(u-v)), u-v \rangle \,dt - \langle F'(v), u-v \rangle \\
&= \int_0^1 \frac{1}{t} \langle F'(v + t(u-v)) - F'(v), t(u-v) \rangle \,dt. 
\end{split} \end{equation*}
With $u_t = v + t(u-v)$, we see that 
\begin{equation*} \begin{split}
D_F(u,v)
&\stackrel{\eqref{Frechet}}{=} \int_0^1 \frac{1}{t} \int_{\Omega} \left ( |\nabla u_t |^{p-2}\nabla u_t - |\nabla v |^{p-2}\nabla v \right )\cdot \nabla (u_t - v) \,dx dt \\ 
&\leq \int_0^1 \frac{1}{t} \int_{\Omega} \left | |\nabla u_t|^{p-2}\nabla u_t - |\nabla v|^{p-2} \nabla v \right | \left | \nabla (u_t - v) \right | \,dx dt \\
&\stackrel{\eqref{BL_upper}}{\lesssim} \int_0^1 \frac{1}{t} \int_{\Omega} \left ( |\nabla  u_t| + |\nabla v| \right )^{p-2} \left | \nabla (u_t - v) \right |^2 \, dx dt \\
&= \int_0^1 t \int_{\Omega} \left ( \left| \nabla \left( v + t(u - v) \right) \right| + |\nabla v| \right )^{p-2} \left | \nabla (u-v) \right |^{2} \, dx dt.  
\end{split} \end{equation*}
Now, we invoke the inequality
\begin{equation}
\label{vector_ineq}
\frac{t}{2} \left( |\xi| + |\eta| \right) \leq |\xi + t \eta| + |\xi| \leq 2 \left( |\xi| + |\eta| \right), \quad \xi, \eta \in \mathbb{R}^2, \text{ } t \in [0,1],
\end{equation} 
to obtain that 
\begin{equation*}
D_F(u,v) \lesssim \int_0^1 t \,dt \cdot \int_{\Omega} \left ( |\nabla v| + |\nabla (u-v)| \right )^{p-2} \left | \nabla (u-v) \right |^{2} \,dx \approx \| u - v \|_{(\nabla v)}^2.
\end{equation*}
Hence, we proved $D_F (u,v) \lesssim \| u - v \|_{(\nabla v)}^2$.
The inequality $D_F (u,v) \gtrsim \| u - v \|_{(\nabla v)}^2$ can be shown in a similar manner using~\eqref{BL_lower} and~\eqref{vector_ineq}.
\end{proof}

\subsection{Quasi-norm stable decomposition}
The core step in the convergence analysis of additive Schwarz methods typically involves verifying a stable decomposition property; see, e.g.,~\cite[Eq.~(13)]{TX:2002} and~\cite[Assumption~4.1]{Park:2020}.
In this section, we derive a quasi-norm stable decomposition property associated with the space decomposition~\eqref{2L}.
A key distinction of the quasi-norm stable decomposition property considered in this section compared to the existing ones is that, we use the quasi-norm $\| \cdot \|_{(\nabla v)}$ while the existing ones are written in terms of norms.
As~\eqref{Bregman_rough} implies, a power of norm cannot approximate the Bregman distance of $F$ by a multiplicative constant if $p \neq 2$.
Our main insight is that if the quasi-norm can approximate the Bregman distance of $F$ up to a multiplicative constant, i.e., if it satisfies an estimate of the form~\eqref{Bregman}, then we can derive the asymptotic linear convergence of \cref{Alg:ASM} using this property.

We recall that two key ingredients for the stable decomposition analysis for linear elliptic problems are the Poincar\'{e}--Friedrichs inequality and interpolation error estimate; see~\cite[Chapter~3]{TW:2005}.
Therefore, we need to establish these theories with respect to the quasi-norm for the stable decomposition analysis of the $p$-Laplacian.

In \cref{Lem:Poincare_greater,Lem:Poincare_less}, we present quasi-norm Poincar\'{e}--Friedrichs inequalities for the cases $p \in (2, \infty)$ and $p \in (1,2)$, respectively, that are suitable for our purposes; more general results are proven in \cref{Sec:Poincare}.
\begin{lemma}
\label{Lem:Poincare_greater}
    Let $p \in (2, \infty)$ and $v \in S_h (\Omega)$. Assume that every maximal polygonal region $R \subset \Omega$ with $| \nabla v | \neq 0$ satisfies that $\partial R \cap \partial \Omega$ contains an element edge.
    Then there exists a positive constant $C_{p,v}^{\PF}$ such that
    \begin{equation*}
        \int_{\Omega} (|w| + |\nabla v|)^{p-2} |w|^2 \,dx 
        \leq C_{p,v}^{\PF} \| w \|_{(\nabla v)}^2, \quad w \in W_0^{1,p} (\Omega).
    \end{equation*}
    Moreover, if $| \nabla v |$ does not vanish on $\Omega$, then $C_{p,v}^{\PF}$ has an upper bound $\overline{C}_{p,v}^{\PF}$ that is continuous at $v$ in $S_h (\Omega)$.
\end{lemma}

\begin{lemma}
    \label{Lem:Poincare_less}
    Let $p \in (1, 2)$ and $v \in S_h (\Omega)$. Assume that every maximal polygonal region $S \subset \Omega$ with $| \nabla v | = 0$ satisfies that $\partial S \cap \partial \Omega$ contains an element edge.
    Then there exists a positive constant $C_{p,v}^{\PF}$ such that
    \begin{equation*}
        \int_{\Omega} (|w| + |\nabla v|)^{p-2} | w |^2 \,dx 
        \leq C_{p,v}^{\PF} \| w \|_{(\nabla v)}^2, \quad w \in W_0^{1,p} (\Omega).
    \end{equation*}
    Moreover, if $| \nabla v |$ does not vanish on $\Omega$, then $C_{p,v}^{\PF}$ has an upper bound $\overline{C}_{p,v}^{\PF}$ that is continuous at $v$ in $S_h (\Omega)$.
\end{lemma}

As stated in \cref{Lem:Poincare_greater,Lem:Poincare_less}, the quasi-norm Poincar\'{e}--Friedrichs inequality holds for all choices of $v$ except for certain exceptional cases, which are detailed in \cref{Ex:counterexample_greater,Ex:counterexample_less}.
Moreover, in most cases, the constant $C_{p,v}^{\PF}$ demonstrates only a weak dependence on $v$.
By the quasi-monotone argument~\cite{GE:2010,PS:2013} presented in \cref{Sec:Poincare}, we can ensure that the value of $C_{p,v}^{\PF}$ is influenced by the local variation of $| \nabla v |$ only.
Consequently, even if $| \nabla v |$ exhibits significant global variation, $C_{p,v}^{\PF}$ has a moderate value.
One may refer to~\cite{SVZ:2012} for relevant numerical evidences.

\begin{remark}
    \label{Rem:regularization}
    As noted in \cref{Lem:Poincare_greater,Lem:Poincare_less}, the quasi-norm Poincare--Friedrichs inequality may not hold in cases where $\nabla v$ vanishes in a certain pattern, which makes the convergence analysis of the algorithm challenging.
    In order to address this issue, one may consider regularization techniques as described in~\cite{DFTW:2020,LCZ:2021}.
    However, we do not adopt such techniques since they require a delicate convergence analysis for the case when the regularization parameter tends to $0$.
\end{remark}

Next, we establish a quasi-norm error estimate for the nodal interpolation operator $I_h$ onto the finite element space $S_h (\Omega)$, as summarized in \cref{Lem:interpolation}.

\begin{lemma}
\label{Lem:interpolation}
Let $w \in W_0^{1,p} (\Omega)$ be a continuous, piecewise quadratic function defined on $\cT_h$ and let $v \in S_h (\Omega)$.
Then, there exists a positive constant $C$, independent of $w$, $v$, and $h$, such that
\begin{equation*}
    \|I_h w\|_{(\nabla v)} \leq C \| w \|_{(\nabla v)}.
\end{equation*}
\end{lemma}
\begin{proof}
Take any $T \in \cT_h$.
We first prove that $I_h$ achieves the local $W^{1,1}$-stability; invoking the inverse inequality~\cite[Lemma~12.1]{EG:2021} and the $H^1$-stability~\cite[Lemma~3.9]{TW:2005} yields
\begin{equation}
\label{W11_stable}
| I_h w |_{W^{1,1}(T)}
\lesssim h | I_h w |_{H^1 (T)}
\lesssim h | w |_{H^1 (T)}
\lesssim | I_h w |_{W^{1,1}(T)}.
\end{equation}
Now, we proceed similarly as in the proof of~\cite[Theorem~4.5]{DR:2007}.
Recall that $| \nabla v |$ is constant on $T$, say $a = |\nabla v| \geq 0$.
Since the map $x \mapsto (x + a)^{p-2} x^2$~($x \geq 0$) is increasing and convex, we have
\begin{equation} \begin{split}
\label{stable_local}
\int_T (| \nabla (I_h w) | &+ a)^{p-2} | \nabla (I_h w)|^2 \,dx \\
&\stackrel{\text{(i)}}{\lesssim} \int_T \left( h^{-2} \int_T | \nabla (I_h w) (y) | \,dy + a \right)^{p-2} \left( h^{-2} \int_T |\nabla (I_h w)| \,dy \right)^2 \,dx \\
&\stackrel{\eqref{W11_stable}}{\lesssim} h^2 \left( h^{-2} \int_T | \nabla w | \,dy + a \right)^{p-2} \left( h^{-2} \int_T | \nabla w | \,dy \right)^2 \\
&\stackrel{\text{(ii)}}{\lesssim} \int_T ( | \nabla w | + a )^{p-2} | \nabla w|^2 \,dy
\end{split}\end{equation}
    where (i) is due to the inverse inequality~(cf.~\cite[Eq.~(2.4)]{DR:2007})
\begin{equation*}
| u |_{W^{1, \infty} (T)} \lesssim h^{-2} | u |_{W^{1, 1} (T)}, \quad u \in S_h (\Omega),
\end{equation*}
and (ii) is due to the Jensen inequality.
By summing~\eqref{stable_local} over all elements $T \in \cT_h$, we arrive at the conclusion. 
\end{proof}

Using \cref{Lem:Poincare_greater,Lem:Poincare_less,Lem:interpolation}, we obtain the following quasi-norm stable decomposition result.

\begin{lemma}
\label{Lem:stable}
Suppose that $p \in (1, \infty)$ and $v \in V$ satisfy either the assumptions in \cref{Lem:Poincare_greater} or those in \cref{Lem:Poincare_less}.
Then there exists a positive constant $C_{p,v}^{\SD}$ depending on $p$ and $v$ such that the following holds:
for any $w \in V$, there exist $w_k \in V_k$, $0 \leq k \leq N$, such that
\begin{eqnarray*}
w = \sum_{k=0}^N R_k^* w_k, \\
\sum_{k=0}^N D_F (v + R_k^* w_k, v) \leq C_{p, v}^{\SD} \| w \|_{(\nabla v)}^2, \\
C_{p,v}^{\SD} \lesssim 1 + C_{p, v}^{\PF} \left( \frac{H}{\delta} \right)^p,
\end{eqnarray*}
where the constant $C_{p,v}^{\PF}$ was given in either \cref{Lem:Poincare_greater} or \cref{Lem:Poincare_less}.
\end{lemma}
\begin{proof}
Throughout this proof, let an index $k$ runs from $1$ to $N$.
Take any $u, v \in V$, and let $w = u - v$.
We define $w_0 \in V_0$ and $w_k \in V_k$ as
\begin{equation*}
    w_0 = \tilde{I}_H w, \quad
    R_k^* w_k = I_h ( \theta_k \tw ),
\end{equation*}
where $\tilde{I}_H$ is the Scott--Zhang quasi-interpolation operator onto $S_H (\Omega)$~\cite{SZ:1990}, and  $\tw = w - R_0^* w_0$.
It is clear that $w = R_0^* w_0 + \sum_{k=1}^N R_k^* w_k$.
Invoking \cref{Lem:Bregman} and~\cite[Theorem~4.5]{DR:2007}, we have
\begin{equation}
    \label{Lem:stable_1}
    D_F ( v + R_0^* w_0, v) \lesssim \| \tilde{I}_H w \|_{(\nabla v)}^2
    \lesssim \| w \|_{(\nabla v)}^2.
\end{equation}
Similarly, \cref{Lem:Bregman,Lem:interpolation} imply
\begin{equation}
    \label{Lem:stable_2}
    D_F (v + R_k^* w_k, v)
    \lesssim \sum_{k=1}^N \| I_h (\theta_k \tw) \|_{(\nabla v)}^2
    \lesssim \sum_{k=1}^N \| \theta_k \tw \|_{(\nabla v)}^2.
\end{equation}
Note that the map $x \mapsto (x + a)^{p-2}x^2$~($x \geq 0$) is increasing for any $a \geq 0$.
It follows that
\begin{equation} \begin{split}
    \label{Lem:stable_3}
    \| \theta_k &\tw \|_{(\nabla v)}^2
    \leq \int_{\Omega} \left( \theta_k |\nabla \tw | + |\tw| |\nabla \theta_k| + |\nabla v | \right)^{p-2} \left( \theta_k |\nabla \tw | + |\tw| |\nabla \theta_k| \right)^2 \,dx \\
    &\stackrel{\text{(i)}}{\lesssim} \int_{\Omega} \left( \theta_k |\nabla \tw| + |\nabla v| \right)^{p-2} \left( \theta_k |\nabla \tw | \right)^2 \,dx
    + \int_{\Omega} \left( |\tw| |\nabla \theta_k| + |\nabla v| \right)^{p-2} \left( |\tw | |\nabla \theta_k| \right)^2 \,dx \\
    &\stackrel{\text{(ii)}}{\lesssim} \| \tw \|_{(\nabla v)}^2
    + \frac{1}{\delta^p} \int_{\Omega} \left( |\tw| + |\nabla v| \right)^{p-2} |\tw|^2 \,dx \\
    &\stackrel{\text{(iii)}}{\leq} \left( 1 + \frac{C_{p, v}^{\PF}}{\delta^p} \right) \| \tw \|_{( \nabla v)}^2,
\end{split} \end{equation}
where (i) is because of the triangle inequality-like result presented in~\cite[Lemma~5.4]{LY:2001}, (ii) is due to~\eqref{pou}, and (iii) is due to \cref{Lem:Poincare_greater,Lem:Poincare_less}.
Meanwhile, we observe that~\cite[Theorem~4.6]{DR:2007} implies
\begin{equation}
    \label{Lem:stable_4}
    \| \tw \|_{(\nabla v)}^2  = \| w - \tilde{I}_H w \|_{(\nabla v)}^2 \lesssim H^p \| w \|_{(\nabla v)}^2.
\end{equation}
Combining~\eqref{Lem:stable_1},~\eqref{Lem:stable_2},~\eqref{Lem:stable_3}, and~\eqref{Lem:stable_4} yields the desired result.
\end{proof}


\begin{remark}
\label{Rem:trace}
The estimate presented in \cref{Lem:stable} is not as sharp as the one in the norm-stable decomposition result given in \cref{Lem:2L}.
Specifically, in \cref{Lem:stable}, the power of $H/\delta$ is $p$, while in \cref{Lem:2L}, it is $p-1$.
The norm-stable decomposition achieves the sharp $(H/\delta)^{p-1}$-result by using a trace theorem-type argument introduced in~\cite{DW:1994}; see also~\cite[Lemma~3.10]{TW:2005}.
Unfortunately, we were unable to make a similar argument in our quasi-norm analysis because the quasi-norm does not have a notion of trace.
To obtain a sharp estimate, it will be necessary to define an appropriate trace for the quasi-norm, which is remained as a topic for future research.
\end{remark}

\subsection{Proof of \texorpdfstring{\cref{Thm:linear}}{Theorem 2.4}}
The proof of \cref{Thm:linear} presented here uses a similar argument to~\cite{Park:2020}.
However, due to the nonlinearity of the quasi-norm $\| \cdot \|_{(\nabla v)}$, we have to make a careful consideration on dealing with $\| \cdot \|_{(\nabla v)}$.
In \cref{Lem:ASM}, we state the generalized additive Schwarz lemma~(see~\cite[Lemma~4.5]{Park:2020}) applied to \cref{Alg:ASM} in a form suitable for our purposes.

\begin{lemma}[generalized additive Schwarz lemma]
\label{Lem:ASM}
Let $\{ u^{(n)} \}$ be the sequence generated by \cref{Alg:ASM}.
Then it satisfies
\begin{equation*}
    u^{(n+1)} \in \argmin_{u \in V} \left\{ F (u^{(n)}) + \langle F'(u^{(n)}), u-u^{(n)} \rangle + M_{\tau} (u, u^{(n)}) \right\},
\end{equation*}
where the functional $M_{\tau} \colon V \times V \rightarrow \mathbb{R}$ is given by
\begin{equation*}
    M_{\tau}(u,v) = \tau \inf \left\{ \sum_{k=0}^N D_F ( v + R_k^* w_k, v) : u-v = \tau \sum_{k=0}^N R_k^* w_k, \text{ } w_k \in V_k \right\}, \quad u,v \in V.
\end{equation*}
\end{lemma}
\begin{proof}
Here, we provide a simple proof that does not rely on convex analysis tools.
We define a functional $Q_n \colon V \rightarrow \mathbb{R}$ as
\begin{equation*}
    Q_n (u) = F( u^{(n)} ) + \langle F' (u^{(n)}), u - u^{(n)} \rangle + M_{\tau} (u, u^{(n)} ), \quad u \in V.
\end{equation*}
For any $u \in V$, invoking~\eqref{Bregman_distance} with some direct computation yields
\begin{equation*}
\begin{split}
    Q_n (u)
    &= (1 - \tau N) F(u^{(n)}) + \tau \inf \left\{ \sum_{k=0}^N F (u^{(n)} + R_k^* w_k ) : u - u^{(n)} = \tau \sum_{k=0}^N R_k^* w_k, \text{ } w_k \in V_k \right\} \\
    &\geq (1 - \tau N) F( u^{(n)}) + \tau \sum_{k=0}^N \min_{w_k \in V_k} F( u^{(n)} + R_k^* w_k ) \\
    &= (1 - \tau N) F(u^{(n)}) + \tau \sum_{k=0}^N F( u^{(n)} + R_k^* w_k^{(n+1)}) \\
    &\geq Q_n (u^{(n+1)}).
\end{split}
\end{equation*}
That is, $u^{(n+1)}$ minimizes $Q_n$, which is our desired result.
\end{proof}

In the following, similar to~\cite[Lemma~4.6]{Park:2020}, we prove that $M_{\tau}$ defined in \cref{Lem:ASM} is bounded below by $D_F$ and above by $\| \cdot \|_{(\nabla v)}^2$ up to a multiplicative constant.
We note that \cref{Lem:M_bound} can be regarded as a variant of the Lipschitz-like/Convexity condition discussed in~\cite{Teboulle:2018}.

\begin{lemma}
\label{Lem:M_bound}
Suppose that $p \in (1, \infty)$ and $v \in V$ satisfy either the assumptions in \cref{Lem:Poincare_greater} or those in \cref{Lem:Poincare_less}.
Then we have
\begin{equation*}
    D_F(u,v) \leq M_{\tau}(u,v) \leq \frac{C_{p,v}^{\SD}}{\tau^{\max \{p, 2\} - 1}} \| u -v \|_{(\nabla v)}^2, \quad u \in V,
\end{equation*}
where $C_{p, v}^{\SD}$ and $M_{\tau}$ were defined in \cref{Lem:ASM,Lem:stable}, respectively.
\end{lemma}
\begin{proof}
We define $\bar{u} \in V$ by the following:
\begin{equation*}
    \bar{u}-v = \frac{1}{\tau} (u-v).
\end{equation*}
By \cref{Lem:stable}, there exist $w_k \in V_k$, $0 \leq k \leq N$, such that
\begin{equation*}
\begin{split}
    \bar{u} - v = \sum_{k=0}^N R_k^* w_k, \\
    \sum_{k=0}^N D_F (v + R_k^* w_k, v) \leq C_{p, v}^{\SD} \| \bar{u} - v \|_{(\nabla v)}^2.
\end{split}
\end{equation*}
Note that $u - v = \tau \sum_{k=0}^N R_k^* w_k$.
It follows by \cref{Lem:scaling} that
\begin{equation*}
    \tau \sum_{k=0}^N D_F (v + R_k^* w_k, v) \leq \tau C_{p, v}^{\SD} \| \bar{u} - v \|_{(\nabla v)}^2 \leq \frac{C_{p, v}^{\SD}}{\tau^{\max \{ p, 2 \} - 1}} \\| u - v \|_{(\nabla v)}^2.
\end{equation*}
Meanwhile, invoking \cref{Lem:convex} yields
\begin{equation*}
    \tau \sum_{k=0}^N D_F(v + R_k^* w_k, v) = \tau \sum_{k=0}^N F(v + R_k^* w_k) - \tau (N+1) F(v) - \langle F'(v), u-v \rangle \geq D_F (u,v),
\end{equation*}
which completes the proof.
\end{proof}

By closely following the argument in~\cite[Appendix~A.4]{Park:2020} and manipulating $\| \cdot \|_{(\nabla v)}$-terms using the properties of $\| \cdot \|_{(\nabla v)}$ presented in \cref{Lem:scaling,Lem:symmetry,Lem:Bregman}, we establish the following lemma, which provides an estimate for the ratio of two consecutive energy errors in \cref{Alg:ASM}.

\begin{lemma}
\label{Lem:ratio}
In \cref{Alg:ASM}, suppose that $v = u^{(n)}$ satisfy either the assumptions in \cref{Lem:Poincare_greater} or those in \cref{Lem:Poincare_less} for some $n \geq 0$.
Then we have
\begin{equation*}
\frac{F(u^{(n+1)}) - F(u^*)}{F(u^{(n)}) - F(u^*)} \leq 1 - \left( 1 - \frac{1}{\up} \right) \left( \frac{\tau^{\op-1} \mu_p }{\up 2^{\hp} C_{p, u^{(n)}}^{\SD}} \right)^{\frac{1}{\up-1}},
\end{equation*}
where $\up = \min \{ p, 2\}$, $\op = \max \{p, 2 \}$, $\hp = |p - 2|$, and $\mu_p$ and $C_{p, u^{(n)}}^{\SD}$ were given in \cref{Lem:Bregman,Lem:stable}, respectively.
\end{lemma}
\begin{proof}
For $t \in [0,1]$, we write
\begin{equation}
    \label{ut}
    u_t = u^{(n)} + t(u^* - u^{(n)}).
\end{equation}
Then we have
\begin{equation*} \begin{split}
    F&(u^{(n+1)}) = F(u^{(n)}) + \langle F'(u^{(n)}), u^{(n+1)} - u^{(n)} \rangle + D_F (u^{(n+1)}, u^{(n)}) \\
    &\stackrel{\text{(i)}}{\leq} F(u^{(n)}) + \langle F'(u^{(n)}), u^{(n+1)} - u^{(n)} \rangle + M_{\tau} (u^{(n+1)}, u^{(n)}) \\
    &\stackrel{\text{(ii)}}{=} \min_{u \in V} \left\{ F(u^{(n)}) + \langle F'(u^{(n)}), u - u^{(n)} \rangle + M_{\tau} (u, u^{(n)}) \right\} \\
    &\stackrel{\text{(i)}}{\leq} \min_{t \geq 0} \left\{ F(u^{(n)}) + \langle F'(u^{(n)}), u_t - u^{(n)} \rangle + \frac{C_{p, u^{(n)}}^{\SD}}{\tau^{\op - 1}} \| u_t - u^{(n)} \|_{(\nabla u^{(n)})}^2 \right\} \\
    &\stackrel{\text{(iii)}}{\leq} \min_{t \geq 0} \left\{ F(u^{(n)}) + t \langle F'(u^{(n)}), u^* - u^{(n)} \rangle + \frac{C_{p, u^{(n)}}^{\SD} }{\tau^{\op - 1}} t^{\up } \| u^* - u^{(n)} \|_{(\nabla u^{(n)})}^2 \right\} \\
    &\stackrel{\text{(iv)}}{\leq} \min_{t \geq 0} \left\{ F(u^{(n)}) + t \langle F'(u^{(n)}), u^* - u^{(n)} \rangle + \frac{2^{\hp} C_{p, u^{(n)}}^{\SD} }{\tau^{\op - 1}} t^{\up} \| u^{(n)} - u^* \|_{(\nabla u^*)}^2 \right\} \\
    &\stackrel{\text{(v)}}{\leq} \min_{t \geq 0} \left\{ F(u^{(n)}) + t \langle F'(u^{(n)}), u^* - u^{(n)} \rangle + \frac{2^{\hp} C_{p, u^{(n)}}^{\SD} }{\tau^{\op - 1} \mu_p}  t^{\up} \left(F(u^{(n)}) - F(u^*) \right) \right\},
\end{split} \end{equation*}
where (i)--(v) are due to Lemmas~\ref{Lem:M_bound}, \ref{Lem:ASM}, \ref{Lem:scaling}, \ref{Lem:symmetry}, and \ref{Lem:Bregman}, respectively.
By the convexity of $F$, we get
\begin{equation*}
\begin{split}
    F(u^{(n+1)}) - F(u^*) &\leq \min_{t \geq 0} \left\{ 1 - t + \frac{2^{\hp} C_{p, u^{(n)}}^{\SD} }{\tau^{\op - 1} \mu_p} t^{\up} \right\} \left(F(u^{(n)}) - F(u^*) \right) \\
    &= \left[ 1 - \left( 1 - \frac{1}{\up} \right) \left( \frac{\tau^{\op-1} \mu_p }{\up 2^{\hp} C_{p, u^{(n)}}^{\SD} } \right)^{\frac{1}{\up-1}} \right] \left(F(u^{(n)}) - F(u^*) \right),
\end{split}
\end{equation*}
which completes the proof.
\end{proof}

Finally, we are ready to present our proof of \cref{Thm:linear}.

\begin{proof}[Proof of \cref{Thm:linear}]
Suppose that $| \nabla u^* |$ does not vanish on $\Omega$.
Take any $\epsilon > 0$.
By the continuity of $\overline{C}_{p, u^*}^{\mathrm{PF}}$ stated in \cref{Lem:Poincare_greater,Lem:Poincare_less}, we can find a neighborhood $B_{\epsilon}$ of $u^*$ in $V$ such that, for any $v \in B_{\epsilon}$, $| \nabla v |$ does not vanish on $\Omega$ and 
\begin{equation}
\label{B_epsilon}
\overline{C}_{p, v}^{\mathrm{PF}} \leq \overline{C}_{p, u^*}^{\mathrm{PF}} + \epsilon.
\end{equation}
Meanwhile, by \cref{Prop:Park} and~\eqref{Bregman_rough}~(see~\cite[Section~6.1]{Park:2020} for the precise statement of~\eqref{Bregman_rough}), we deduce that the sequence $\{ u^{(n)} \}$ converges to $u^*$ in $V$.
Hence, there exists $n_0$ such that, if $n \geq n_0$, then $u^{(n)} \in B_{\epsilon}$.
By \cref{Lem:stable}, for $n \geq n_0$, we have
\begin{equation*}
C_{p, u^{(n)}}^{\SD} \lesssim 1 + C_{p, u^{(n)}}^{\PF} \left( \frac{H}{\delta} \right)^p
\leq 1 + \overline{C}_{p, u^{(n)}}^{\PF} \left( \frac{H}{\delta} \right)^p
\stackrel{\eqref{B_epsilon}}{\leq} 1 + \left( \overline{C}_{p, u^*}^{\PF} + \epsilon \right) \left( \frac{H}{\delta} \right)^p.
\end{equation*}
It follows by \cref{Lem:ratio} that
\begin{equation*}
\begin{split}
\left( 1 - \frac{F(u^{(n+1)}) - F(u^*)}{F(u^{(n)}) - F(u^*)} \right)^{-1}
&\leq \frac{\up}{\up - 1} \left( \frac{\up 2^{\hp} C_{p, u^{(n)}}^{\SD} }{\tau^{\op - 1} \mu_p} \right)^{\frac{1}{\up - 1}} \\
&\lesssim \frac{\up}{\up - 1} \left( \frac{\up  2^{\hp}}{\tau^{\op - 1} \mu_p } \right)^{\frac{1}{\up - 1}}\left[ 1 + \left( \overline{C}_{p, u^*}^{\PF} + \epsilon \right) \left( \frac{H}{\delta} \right)^p \right]^{\frac{1}{\up - 1}}.
\end{split}
\end{equation*}
Since $\epsilon$ is arbitrary, we obtain the desired result with
\begin{equation*}
    \gamma = \frac{\up}{\up - 1} \left( \frac{\up  2^{\hp}}{\tau^{\op - 1} \mu_p } \right)^{\frac{1}{\up - 1}}\left[ 1 +  \overline{C}_{p, u^*}^{\PF} \left( \frac{H}{\delta} \right)^p \right]^{\frac{1}{\up - 1}}.
\end{equation*}
This completes the proof.
\end{proof}

\begin{remark}
\label{Rem:Bregman}
As stated in \cref{Lem:Bregman}, the squared quasi-norm $\| u - v \|_{(\nabla v)}^2$ is equivalent to the Bregman distance $D_F (u, v)$ up to multiplicative constants for $u, v \in W^{1,p} (\Omega)$.
This equivalence allows us to perform the convergence analysis presented in this section using the Bregman distance instead of the quasi-norm.
Indeed, the Bregman distance is frequently utilized in the analysis of convex optimization algorithms; see, e.g.,~\cite{BBT:2017,Teboulle:2018}.
Nevertheless, we opted to present the convergence analysis of \cref{Alg:ASM} in terms of the quasi-norm in this paper, because using the quasi-norm has an advantage that we can simplify our proof by borrowing some useful techniques regarding the quasi-norm introduced in from the existing literature~\cite{CLY:2006,EL:2005,LY:2001,LY:2002}.
\end{remark}

\section{Quasi-norm Poincar\'{e}--Friedrichs inequality}
\label{Sec:Poincare}
This section is devoted to the proofs of \cref{Lem:Poincare_greater,Lem:Poincare_less}.
Namely, we deal with quasi-norm Poincar\'{e}--Friedrichs inequalities of the form
\begin{equation}
    \label{Poincare_ideal}
    \int_{\Omega} (|w| + |\nabla v|)^{p-2} |w|^2 \,dx \leq C \| w \|_{(\nabla v)}^2,
\end{equation}
where $p \in (1, \infty)$ with $p \neq 2$.
Unfortunately, the inequality~\eqref{Poincare_ideal} does not hold for every $v, w  \in W_0^{1,p} (\Omega)$; see \cref{Ex:counterexample_greater,Ex:counterexample_less}.
Based on a quasi-monotonicity argument introduced in~\cite{PS:2013}, we characterize the conditions on $v$ such that the inequality~\eqref{Poincare_ideal} holds and provide a precise estimate for the Poincar\'{e}--Friedrichs constant $C$ in~\eqref{Poincare_ideal}.
Throughout this section, let $W_{\Gamma}^{1,p} (\Omega)$ denote the collection of all $W^{1,p} (\Omega)$-functions vanishing on $\Gamma \subset \partial \Omega$.
In addition, we use the conventions $0/0 = 1$ and $1/0 = \infty$.

We first observe that a particular case of~\eqref{Poincare_ideal}, when $|\nabla v|$ is constant on $\Omega$, is valid.
By the same argument as in~\cite[Lemma~3.1]{LY:2002} and~\cite[Lemma.~4.1]{CLY:2006}, we can prove the following lemma.

\begin{lemma}
    \label{Lem:Poincare_constant}
    Let $\Gamma \subset \partial \Omega$ have nonvanishing one-dimensional measure.
    Then, there exist a positive constant $C$ such that
    \begin{equation*}
        \int_{\Omega} ( |w | + a )^{p-2} |w|^2 \,dx \leq
        C \int_{\Omega} (| \nabla w | + a )^{p-2} |\nabla w|^2 \,dx,
        \quad w \in W_{\Gamma}^{1,p} (\Omega), \text{ } a \geq 0.
    \end{equation*}
\end{lemma}

For a nonnegative function $\alpha \in L^{\infty} (\Omega)$ and a partition $\mathcal{Y} = \{ Y \}_{l=1}^m$ of $\Omega$ consisting of nonoverlapping polygonal regions, we define two nonnegative functions 
$\underline{\alpha}_{\mathcal{Y}}, \overline{\alpha}^{\mathcal{Y}} \in L^{\infty} (\Omega)$ as follows:
\begin{equation*}
    \underline{\alpha}_{\mathcal{Y}} (x) = \operatornamewithlimits{ess\inf}_{Y_l} \alpha, \quad
    \overline{\alpha}^{\mathcal{Y}} (x) =
    \operatornamewithlimits{ess\sup}_{Y_l} \alpha, \quad
    x \in Y_{l}, \text{ } 1 \leq l \leq m.
\end{equation*}

In the following, we address the cases $p \in (2, \infty)$ and $p \in (1, 2)$ separately.
We first focus on the case $p \in (2, \infty)$.
In \cref{Def:quasi_increase}, we introduce the concept of \textit{quasi-monotone increase}.
We note that relevant notions were explored in~\cite{GE:2010,PS:2013}.

\begin{definition}
\label{Def:quasi_increase}
Let $\alpha \in L^{\infty} (\Omega)$ be a nonnegative function on $\Omega$, and let $\mathcal{Y} = \{ Y \}_{l=1}^m$ denote a partition of $\Omega$ into nonoverlapping polygonal regions.
\begin{enumerate}
    \item We say that the region $P_
    {l_1, l_s} = ( \overline{Y}_{l_1} \cup \dots \cup \overline{Y}_{l_s} )^{\circ}$, $1 \leq l_1, \dots, l_s \leq m$, is a quasi-monotonically increasing path from $Y_{l_1}$ to $Y_{l_s}$ with respect to $\alpha$ if the following two conditions hold:
    \begin{enumerate}
        \item For each $1 \leq i \leq s-1$, the regions $\overline{Y}_{l_i}$ and $\overline{Y}_{l_{i+1}}$ share a common edge.
        \item $\underline{\alpha}_{\mathcal{Y}} (Y_{l_1}) \leq \dots \leq \underline{\alpha}_{\mathcal{Y}} (Y_{l_s})$.
    \end{enumerate}
    \item We say that $\alpha$ is $\partial \Omega$-quasi-monotonically increasing on $\mathcal{Y}$ if, for any $1 \leq l \leq m$, there exist an index $l^*$ and a quasi-monotonically increasing path $P_{l, l^*}$ from $Y_{l}$ to $Y_{l^*}$, such that $\partial Y_{l^*} \cap \partial \Omega$ has nonvanishing one-dimensional measure.
\end{enumerate}
\end{definition}

By a similar argument as in the proof of~\cite[Theorem~2.9]{PS:2013}, we prove the following lemma.

\begin{lemma}
\label{Lem:quasi_increase}
    Assume that $p \in (2, \infty)$.
    Let $\alpha \in L^{\infty} (\Omega)$ be a nonnegative function on $\Omega$, and let $\mathcal{Y} = \{ Y \}_{l=1}^m$ denote a partition of $\Omega$ into nonoverlapping polygonal regions.
    If $\alpha$ is $\partial \Omega$-quasi-monotonically increasing on $\mathcal{Y}$, then, for each $1 \leq l \leq m$, there exists a positive constant $c_{p, \mathcal{Y}, l}$, independent of $\alpha$, such that
    \begin{equation*}
        \int_{Y_l} ( |w| + \alpha )^{p-2} |w|^2 \,dx 
        \leq c_{p, \mathcal{Y}, l} \left( \frac{\overline{\alpha}^{\mathcal{Y}} (Y_l)}{\underline{\alpha}_{\mathcal{Y}} (Y_l)} \right)^{p-2} \int_{P_{l, l^*}} ( | \nabla w| + \alpha )^{p-2} |\nabla w|^2 \,dx,
         \text{ } w \in W_0^{1,p} (\Omega),
    \end{equation*}
\end{lemma}
where the region $P_{l,l^*}$ was given in \cref{Def:quasi_increase}.
\begin{proof}
Note that the map $x \mapsto x^{p-2}$~($x \geq 0$) is increasing.
Take any $l$ such that $1 \leq l \leq m$.
Since $\alpha \leq \overline{\alpha}^{\mathcal{Y}} (Y_l)$ on $Y_l$ and $Y_l \subset P_{l, l^*}$, we get
\begin{equation}
\label{Lem:quasi_increase_1}
\begin{split}
    \int_{Y_l} (|w| + \alpha)^{p-2} |w|^2 \,dx
    &\leq \int_{Y_l} \left( |w| + \overline{\alpha}^{\mathcal{Y}} (Y_l) \right)^{p-2} |w|^2 \,dx \\
    &\leq \int_{P_{l, l^*}} \left( |w| + \overline{\alpha}^{\mathcal{Y}} (Y_l) \right)^{p-2} |w|^2 \,dx.
\end{split}
\end{equation}
Since $\partial P_{l, l^*} \cap \partial \Omega$ has nonvanishing one-dimensional measure, \cref{Lem:Poincare_constant} ensures that we have a positive constant $c_{p, \mathcal{Y}, l}$, independent of $w$ and $\alpha$, such that
\begin{equation}
\label{Lem:quasi_increase_2}
    \int_{P_{l, l^*}} \left( |w| + \overline{\alpha}^{\mathcal{Y}} (Y_l) \right)^{p-2} |w|^2 \,dx
    \leq c_{p, \mathcal{Y}, l} \int_{P_{l, l^*}} \left( |\nabla w| + \overline{\alpha}^{\mathcal{Y}} (Y_l) \right)^{p-2} |\nabla w|^2 \,dx.
\end{equation}
Invoking the inequality
\begin{equation*}
    \frac{a + \overline{\alpha}^{\mathcal{Y}} (Y_l)}{a + \underline{\alpha}_{\mathcal{Y}} (Y_l)} \leq \frac{\overline{\alpha}^{\mathcal{Y}} (Y_l)}{\underline{\alpha}_{\mathcal{Y}} (Y_l)},
    \quad a \geq 0,
\end{equation*}
we have
\begin{equation}
\label{Lem:quasi_increase_3}
\begin{split}
    \int_{P_{l, l^*}} &\left( |\nabla w| + \overline{\alpha}^{\mathcal{Y}} (Y_l) \right)^{p-2} |\nabla w|^2 \,dx \\
    &\leq \left( \frac{\overline{\alpha}^{\mathcal{Y}} (Y_l)}{\underline{\alpha}_{\mathcal{Y}} (Y_l)} \right)^{p-2} \int_{P_{l, l^*}} \left( |\nabla w| + \underline{\alpha}_{\mathcal{Y}} (Y_l) \right)^{p-2} |\nabla w|^2 \,dx \\
    &\leq \left( \frac{\overline{\alpha}^{\mathcal{Y}} (Y_l)}{\underline{\alpha}_{\mathcal{Y}} (Y_l)} \right)^{p-2} \int_{P_{l, l^*}} \left( |\nabla w| + \underline{\alpha}_{\mathcal{Y}} \right)^{p-2} |\nabla w|^2 \,dx \\
    &\leq \left( \frac{\overline{\alpha}^{\mathcal{Y}} (Y_l)}{\underline{\alpha}_{\mathcal{Y}} (Y_l)} \right)^{p-2} \int_{P_{l, l^*}} \left( |\nabla w| + \alpha \right)^{p-2} |\nabla w|^2 \,dx,
\end{split}
\end{equation}
where the penultimate inequality is because $\underline{\alpha}_{\mathcal{Y}}$ increases along $P_{l,l^*}$.
Combining Lemmas~\ref{Lem:quasi_increase_1},~\ref{Lem:quasi_increase_2}, and~\ref{Lem:quasi_increase_3} yields the desired result.
\end{proof}

Thanks to \cref{Lem:quasi_increase}, we are able to define the \textit{quasi-monotone increase constant} $C_{p, \alpha}^{\QM}$ for $p \in (2, \infty)$, as presented in \cref{Def:quasi_increase_constant}.

\begin{definition}
\label{Def:quasi_increase_constant}
    Assume that $p \in (2, \infty)$.
    Let $\alpha \in L^{\infty} (\Omega)$ be a nonnegative function on $\Omega$.
    The quasi-monotone increase constant $C_{p, \alpha}^{\QM} \in [0, \infty]$ is defined by
    \begin{equation*}
        C_{p, \alpha}^{\QM} = \inf_{\mathcal{Y}} \left\{ \max_{1 \leq l \leq m} \left( \frac{\overline{\alpha}^{\mathcal{Y}} (Y_l)}{\underline{\alpha}_{\mathcal{Y}} (Y_l)} \right)^{p-2} \cdot \sum_{l=1}^m c_{p, \mathcal{Y}, l} \right\},
    \end{equation*}
    where the constants $c_{p, \mathcal{Y}, l}$'s were given in \cref{Lem:quasi_increase} and the infimum is taken over every nonoverlapping polygonal partition $\mathcal{Y} = \{ Y_l \}_{l=1}^m$ of $\Omega$ such that $\alpha$ is $\partial \Omega$-quasi-monotonically increasing on $\mathcal{Y}$.
\end{definition}

Note that the infimum in \cref{Def:quasi_increase} is well-defined because $\alpha$ is $\partial \Omega$-quasi-monotonically increasing on the trivial partition $\{ \Omega \}$.
In terms of the quasi-monotone increase constant $C_{p, \alpha}^{\QM}$, we present a quasi-norm Poincar\'{e}--Friedrichs inequality for $p \in (2, \infty)$ in \cref{Thm:Poincare_greater}.

\begin{theorem}
\label{Thm:Poincare_greater}
    Assume that $p \in (2, \infty)$.
    Let $\alpha \in L^{\infty} (\Omega)$ be a nonnegative function on $\Omega$.
    Then we have
    \begin{equation*}
        \int_{\Omega} (|w| + \alpha )^{p-2} |w|^2 \,dx \leq C_{p, \alpha}^{\QM} \int_{\Omega} (|\nabla w| + \alpha)^{p-2} |\nabla w|^2 \,dx, \quad w \in W_0^{1,p} (\Omega),
    \end{equation*}
    where the quasi-monotone increase constant $C_{p, \alpha}^{\QM} \in [0, \infty]$ was given in \cref{Def:quasi_increase_constant}.
\end{theorem}
\begin{proof}
We fix any nonoverlapping polygonal partition $\mathcal{Y}$ such that $\alpha$ is $\partial \Omega$-quasi-monotonically increasing on $\mathcal{Y}$.
By \cref{Lem:quasi_increase}, we have
\begin{equation}
\label{Thm:Poincare_greater_1}
\begin{split}
    \int_{Y_l} (|w| + \alpha)^{p-2} |w|^2 \,dx
    &\leq c_{p, \mathcal{Y}, l} \left( \frac{\overline{\alpha}^{\mathcal{Y}} (Y_l)}{\underline{\alpha}_{\mathcal{Y}} (Y_l)} \right)^{p-2} \int_{P_{l, l^*}} ( | \nabla w| + \alpha )^{p-2} |\nabla w|^2 \,dx \\
    &\leq c_{p, \mathcal{Y}, l} \max_{1 \leq l \leq m} \left( \frac{\overline{\alpha}^{\mathcal{Y}} (Y_l)}{\underline{\alpha}_{\mathcal{Y}} (Y_l)} \right)^{p-2} \int_{\Omega} ( | \nabla w| + \alpha )^{p-2} |\nabla w|^2 \,dx.
\end{split}
\end{equation}
Summing~\eqref{Thm:Poincare_greater_1} over all $l$ followed by taking the infimum over all $\mathcal{Y}$ completes the proof.
\end{proof}

Let $W_h (\Omega)$ be the space of piecewise constant functions on the triangulation $\cT_h$.
Under an additional assumption that $\alpha \in W_h (\Omega)$, we can characterize the condition when the quasi-monotone increase constant $C_{p, \alpha}^{\QM}$ is finite.

\begin{lemma}
\label{Lem:quasi_increase_constant}
    Assume that $p \in (2, \infty)$.
    Let $\alpha \in W_h (\Omega)$ be a nonnegative piecewise constant function on $\cT_h$.
    Then, the quasi-monotone increase constant $C_{p, \alpha}^{\QM}$ is finite if and only if every maximal polygonal region $R \subset \Omega$ with $\alpha > 0$ satisfies that $\partial R \cap \partial \Omega$ has nonvanishing one-dimensional measure. 
\end{lemma}
\begin{proof}
We first assume that every maximal polygonal region $R \subset \Omega$ with $\alpha > 0$ satisfies that $\partial R \cap \partial \Omega$ has nonvanishing one-dimensional measure.
We consider the partition $\mathcal{Y}^*$ of $\Omega$ consisting of all maximal polygonal regions $\{ R_i \}$ with $\alpha > 0$ and all maximal polygonal regions $\{ S_i \}$ with $\alpha = 0$.
It is obvious that each $R_i$ forms a quasi-monotonically increasing path with respect to $\alpha$ from $R_i$ to itself.
For each $S_i$, if $\partial S_i \cap \partial \Omega$ has nonvanishing one-dimensional measure, then $S_i$ forms a quasi-monotonically increasing path with respect to $\alpha$ from $S_i$ to itself.
Otherwise, the maximality of $S_i$ implies there exists some $R_j$ such that $\overline{S}_i$ and $\overline{R}_j$ share a common edge.
Then we readily deduce that $(\overline{S}_i \cup \overline{R}_j)^{\circ}$ forms a quasi-monotonically increasing path with respect to $\alpha$ from $S_i$ to $R_j$.
Meanwhile, since $\alpha$ is piecewise constant, we have
\begin{equation*}
    \frac{\overline{\alpha}^{\mathcal{Y}^*} (R_i)}{\underline{\alpha}_{\mathcal{Y}^*} (R_i)} < \infty, \quad
    \frac{\overline{\alpha}^{\mathcal{Y}^*} (S_i)}{\underline{\alpha}_{\mathcal{Y}^*} (S_i)} = \frac{0}{0} = 1,
\end{equation*}
for every $R_i$ and $S_i$.
Hence, we conclude that $C_{p, \alpha}^{\QM} < \infty$.

Next, we suppose that there exists a maximal polygonal region $R^* \subset \Omega$ with $\alpha > 0$ such that $\partial R^* \cap \partial \Omega$ is a null set.
That is, every edge of $R^*$ is shared with a region with $\alpha = 0$.
Take any nonoverlapping polygonal partition $\mathcal{Y}$ of $\Omega$.
If $\mathcal{Y}$ has an element $Y$ such that $Y \subset R^*$, then it is impossible to find any quasi-monotonically increasing path with respect to $\alpha$ starting from $Y$, since any such path would necessarily have to pass through a region where $\alpha = 0$.
Otherwise, $\mathcal{Y}$ must contain an element $Y$ such that both $Y \cap R^*$ and $Y \setminus R^*$ are nontrivial, which implies that
\begin{equation*}
    \frac{\overline{\alpha}^{\mathcal{Y}} (Y)}{\underline{\alpha}_{\mathcal{Y}} (Y)} = \infty.
\end{equation*}
Hence, we conclude that $C_{p, \alpha}^{\QM} = \infty$, which completes the proof.
\end{proof}

Combining \cref{Thm:Poincare_greater} and \cref{Lem:quasi_increase_constant} yields \cref{Cor:Poincare_greater}, in which \cref{Lem:Poincare_greater} is a particular case $\alpha =  |\nabla v |$ of this result. 

\begin{corollary}
    \label{Cor:Poincare_greater}
    Assume that $p \in (2, \infty)$.
    Let $\alpha \in W_h (\Omega)$ be a nonnegative piecewise constant function on $\cT_h$.
    If every maximal polygonal region $R \subset \Omega$ with $\alpha > 0$ satisfies that $\partial R \cap \partial \Omega$ has nonvanishing one-dimensional measure, then we have
    \begin{equation*}
        \int_{\Omega} (|w| + \alpha )^{p-2} |w|^2 \,dx \leq C_{p, \alpha}^{\QM} \int_{\Omega} (|\nabla w| + \alpha)^{p-2} |\nabla w|^2 \,dx, \quad w \in W_0^{1,p} (\Omega),
    \end{equation*}
    where $C_{p, \alpha}^{\QM}$ is a finite constant given in \cref{Def:quasi_increase_constant}.
    Moreover, if $\alpha$ does not vanish on $\Omega$, then $C_{p,\alpha}^{\QM}$ has an upper bound $\overline{C}_{p,\alpha}^{\QM}$ that is continuous at $\alpha$ in $W_h (\Omega)$.
\end{corollary}
\begin{proof}
It suffices to find a continuous upper bound of $C_{p, \alpha}^{\QM}$ under the given condition.
We assume that $\alpha$ does not vanish on $\Omega$.
By \cref{Def:quasi_increase_constant}, we have
\begin{equation*}
C_{p, \alpha}^{\QM} \leq  \frac{\max_{\Omega} \alpha}{\min_{\Omega} \alpha} c_{p, \{ \Omega \}, 1}  =: \overline{C}_{p, \alpha}^{\QM},
\end{equation*}
where the inequality is obtained by taking $\mathcal{Y} = \{ \Omega \} = \{ Y_1 \} $ in \cref{Def:quasi_increase_constant}.
As $\alpha > 0$ in $\Omega$, it is clear that $\overline{C}_{p, \alpha}^{\QM}$ is continuous at $\alpha$ in $W_h (\Omega)$, which completes the proof.
\end{proof}

We show that, under the condition presented in \cref{Lem:quasi_increase_constant} for the quasi-monotone increase constant $C_{p, \alpha}$ to be infinite, the quasi-norm Poincar\'{e}--Friedrichs inequality of the form~\eqref{Poincare_ideal} is not valid.
For simplicity, we provide a counterexample in one-dimension; we note that the construction can be generalized to higher dimensions.

\begin{example}
\label{Ex:counterexample_greater}
Let $p \in (2, \infty)$ and $\Omega = (0, 1) \subset \mathbb{R}$.
We define $w \in W_0^{1,p} (\Omega)$ and $\alpha \in L^{\infty} (\Omega)$ as
\begin{equation*}
    w(x) = \begin{cases}
    3x, & \text{ if } 0 < x < \frac{1}{3}, \\
    1 & \text{ if } \frac{1}{3} \leq x < \frac{2}{3}, \\
    -3x+3, & \text{ if } \frac{2}{3} \leq x < 1,
    \end{cases}
    \quad
     \alpha (x) = \begin{cases}
    1, & \text{ if } \frac{1}{3} \leq x < \frac{2}{3}, \\
    0, & \text{ otherwise.}
    \end{cases} \quad
\end{equation*}
We observe that the quasi-monotone increase constant $C_{p, \alpha}$ becomes infinite because the interval $(1/3, 2/3)$ where $\alpha$ is nonzero does not touch $\partial \Omega$.
For any $\epsilon > 0$, direct calculation yields
\begin{equation*}
    \frac{\int_0^1 ( |(\epsilon w)'| + \alpha)^{p-2} |(\epsilon w)'|^2 \,dx}{\int_0^1 (|\epsilon w| + \alpha)^{p-2} |\epsilon w|^2 \,dx} \rightarrow 0
    \quad
    \text{ as } \quad \epsilon \rightarrow 0^+,
\end{equation*}
which implies that~\eqref{Poincare_ideal} does not hold.
\end{example}

We now turn to the case $p \in (1, 2)$.
In contrast to the case $p \in (2, \infty)$, which heavily relies on the quasi-monotone increase of $\alpha$, the analysis of the case $p \in (1, 2)$ hinges on the \textit{quasi-monotone decrease} of $\alpha$; see \cref{Def:quasi_decrease}.

\begin{definition}
\label{Def:quasi_decrease}
Let $\alpha \in L^{\infty} (\Omega)$ be a nonnegative function on $\Omega$, and let $\mathcal{Y} = \{ Y \}_{l=1}^m$ denote a partition of $\Omega$ into nonoverlapping polygonal regions.
\begin{enumerate}
    \item We say that the region $P_{l_1, l_s} = ( \overline{Y}_{l_1} \cup \dots \cup \overline{Y}_{l_s} )^{\circ}$, $1 \leq l_1, \dots, l_s \leq m$, is a quasi-monotonically decreasing path from $Y_{l_1}$ to $Y_{l_s}$ with respect to $\alpha$ if the following two conditions hold:
    \begin{enumerate}
        \item For each $1 \leq i \leq s-1$, the regions $\overline{Y}_{l_i}$ and $\overline{Y}_{l_{i+1}}$ share a common edge.
        \item $\overline{\alpha}_{\mathcal{Y}} (Y_{l_1}) \geq \dots \geq \overline{\alpha}_{\mathcal{Y}} (Y_{l_s})$.
    \end{enumerate}
    \item We say that $\alpha$ is $\partial \Omega$-quasi-monotonically decreasing on $\mathcal{Y}$ if, for any $1 \leq l \leq m$, there exist an index $l^*$ and a quasi-monotonically decreasing path $P_{l, l^*}$ from $Y_{l}$ to $Y_{l^*}$, such that $\partial Y_{l^*} \cap \partial \Omega$ has nonvanishing one-dimensional measure.
\end{enumerate}
\end{definition}

One can prove the following lemma by using the fact that the map $x \mapsto x^{p-2}$~($x \geq 0$) is decreasing and by following a similar argument to that used in the proof of \cref{Lem:quasi_increase}.

\begin{lemma}
\label{Lem:quasi_decrease}
    Assume that $p \in (1, 2)$.
    Let $\alpha \in L^{\infty} (\Omega)$ be a nonnegative function on $\Omega$, and let $\mathcal{Y} = \{ Y \}_{l=1}^m$ denote a partition of $\Omega$ into nonoverlapping polygonal regions.
    If $\alpha$ is $\partial \Omega$-quasi-monotonically decreasing on $\mathcal{Y}$, then, for each $1 \leq l \leq m$, there exists a positive constant $c_{p, \mathcal{Y}, l}$, independent of $\alpha$, such that
    \begin{equation*}
        \int_{Y_l} ( |w| + \alpha )^{p-2} |w|^2 \,dx \\
        \leq c_{p, \mathcal{Y}, l} \left( \frac{\underline{\alpha}_{\mathcal{Y}} (Y_l)}{\overline{\alpha}^{\mathcal{Y}} (Y_l)} \right)^{p-2} \int_{P_{l, l^*}} ( | \nabla w| + \alpha )^{p-2} |\nabla w|^2 \,dx,
         \text{ } w \in W_0^{1,p} (\Omega),
    \end{equation*}
\end{lemma}
where the region $P_{l,l^*}$ was given in \cref{Def:quasi_decrease}.

Similar to \cref{Def:quasi_increase_constant}, we present the definition of the \textit{quasi-monotone decrease constant} $C_{p, \alpha}^{\QM}$ for $p \in (1, 2)$ in the following.

\begin{definition}
\label{Def:quasi_decrease_constant}
    Assume that $p \in (1, 2)$.
    Let $\alpha \in L^{\infty} (\Omega)$ be a nonnegative function on $\Omega$.
    The quasi-monotone decrease constant $C_{p, \alpha}^{\QM} \in [0, \infty]$ is defined by
    \begin{equation*}
        C_{p, \alpha}^{\QM} = \inf_{\mathcal{Y}} \left\{ \max_{1 \leq l \leq m} \left( \frac{\underline{\alpha}_{\mathcal{Y}} (Y_l)}{\overline{\alpha}^{\mathcal{Y}} (Y_l)} \right)^{p-2} \cdot \sum_{l=1}^m c_{p, \mathcal{Y}, l} \right\},
    \end{equation*}
    where the constants $c_{p, \mathcal{Y}, l}$'s were given in \cref{Lem:quasi_decrease} and the infimum is taken over every nonoverlapping polygonal partition $\mathcal{Y}$ of $\Omega$ such that $\alpha$ is $\partial \Omega$-quasi-monotonically decreasing on $\mathcal{Y}$.
\end{definition}

In terms of the quasi-monotone decrease constant $C_{p, \alpha}^{\QM}$, we present a quasi-norm Poincar\'{e}--Friedrichs inequality for $p \in (1, 2)$ in \cref{Thm:Poincare_less}, which can be proven in a similar manner to \cref{Thm:Poincare_greater}.

\begin{theorem}
\label{Thm:Poincare_less}
    Assume that $p \in (1, 2)$.
    Let $\alpha \in L^{\infty} (\Omega)$ be a nonnegative function on $\Omega$.
    Then we have
    \begin{equation*}
        \int_{\Omega} (|w| + \alpha )^{p-2} |w|^2 \,dx \leq C_{p, \alpha}^{\QM} \int_{\Omega} (|\nabla w| + \alpha)^{p-2} |\nabla w|^2 \,dx, \quad w \in W_0^{1,p} (\Omega),
    \end{equation*}
    where the quasi-monotone decrease constant $C_{p, \alpha}^{\QM} \in [0, \infty]$ was given in \cref{Def:quasi_decrease_constant}.
\end{theorem}

The following lemma characterizes the condition when the quasi-monotone decrease constant $C_{p, \alpha}$ is finite, under an additional assumption that $\alpha \in W_h (\Omega)$, i.e., $\alpha$ is a nonnegative piecewise constant function on the triangulation $\cT_h$.

\begin{lemma}
\label{Lem:quasi_decrease_constant}
    Assume that $p \in (1, 2)$.
    Let $\alpha \in W_h (\Omega)$ be a nonnegative piecewise constant function on $\cT_h$.
    Then, the quasi-monotone decrease constant $C_{p, \alpha}^{\QM}$ is finite if and only if every maximal polygonal region $S \subset \Omega$ with $\alpha = 0$ satisfies that $\partial S \cap \partial \Omega$ has nonvanishing one-dimensional measure. 
\end{lemma}
\begin{proof}
The proof is analogous to that of \cref{Lem:quasi_increase_constant}.
\end{proof}

We obtain \cref{Cor:Poincare_less}, in which \cref{Lem:Poincare_less} is a particular case $\alpha = |\nabla u |$, as a direct consequence of \cref{Thm:Poincare_less} and \cref{Lem:quasi_decrease_constant}.

\begin{corollary}
    \label{Cor:Poincare_less}
    Assume that $p \in (1, 2)$.
    Let $\alpha \in W_h (\Omega)$ be a nonnegative piecewise constant function on $\cT_h$.
    If every maximal polygonal region $S \subset \Omega$ with $\alpha = 0$ satisfies that $\partial S \cap \partial \Omega$ has nonvanishing one-dimensional measure, then we have
    \begin{equation*}
        \int_{\Omega} (|w| + \alpha )^{p-2} |w|^2 \,dx \leq C_{p, \alpha}^{\QM} \int_{\Omega} (|\nabla w| + \alpha)^{p-2} |\nabla w|^2 \,dx, \quad w \in W_0^{1,p} (\Omega),
    \end{equation*}
    where $C_{p, \alpha}^{\QM}$ is a finite constant given in \cref{Def:quasi_decrease_constant}.
    Moreover, if $\alpha$ does not vanish on $\Omega$, then $C_{p,\alpha}^{\QM}$ has an upper bound $\overline{C}_{p,\alpha}^{\QM}$ that is continuous at $\alpha$ in $W_h (\Omega)$.
\end{corollary}

Finally, we present a counterexample of the quasi-norm Poincar\'{e}--Friedrichs inequality~\eqref{Poincare_ideal} under the condition presented in \cref{Lem:quasi_decrease_constant} for the quasi-monotone decrease constant $C_{p, \alpha}^{\QM}$ to be infinite.

\begin{example}
\label{Ex:counterexample_less}
Let $p \in (1, 2)$ and $\Omega = (0, 1) \subset \mathbb{R}$.
We define $w \in W_0^{1,p} (\Omega)$ and $\alpha \in L^{\infty} (\Omega)$ as
\begin{equation*}
    w(x) = \begin{cases}
    3x, & \text{ if } 0 < x < \frac{1}{3}, \\
    1 & \text{ if } \frac{1}{3} \leq x < \frac{2}{3}, \\
    -3x+3, & \text{ if } \frac{2}{3} \leq x < 1,
    \end{cases}
    \quad
     \alpha (x) = \begin{cases}
    0, & \text{ if } \frac{1}{3} \leq x < \frac{2}{3}, \\
    1, & \text{ otherwise.}
    \end{cases} \quad
\end{equation*}
We observe that the quasi-monotone increase constant $C_{p, \alpha}^{\QM}$ becomes infinite because the interval $(1/3, 2/3)$ where $\alpha$ vanishes does not touch $\partial \Omega$.
For $\epsilon > 0$, direct calculation yields
\begin{equation*}
    \frac{\int_0^1 ( |(\epsilon w)'| + \alpha)^{p-2} |(\epsilon w)'|^2 \,dx}{\int_0^1 (|\epsilon w| + \alpha)^{p-2} |\epsilon w|^2 \,dx}
    \rightarrow 0 \quad
    \text{ as } \quad \epsilon \rightarrow 0^+,
\end{equation*}
which implies that~\eqref{Poincare_ideal} does not hold.
\end{example}

\section{Numerical experiments}
\label{Sec:Numerical}
In this section, we present numerical results of the two-level additive Schwarz method for the $p$-Laplacian, which support our theoretical findings. 
All the algorithms were implemented in MATLAB R2022b.
They were executed on a desktop equipped with AMD Ryzen~5 5600X CPU (3.7GHz, 6C), 40GB RAM, and the operating system Windows~10 Pro.

\begin{figure}
\centering
  \subfloat[][]{\includegraphics[height=0.28\hsize]{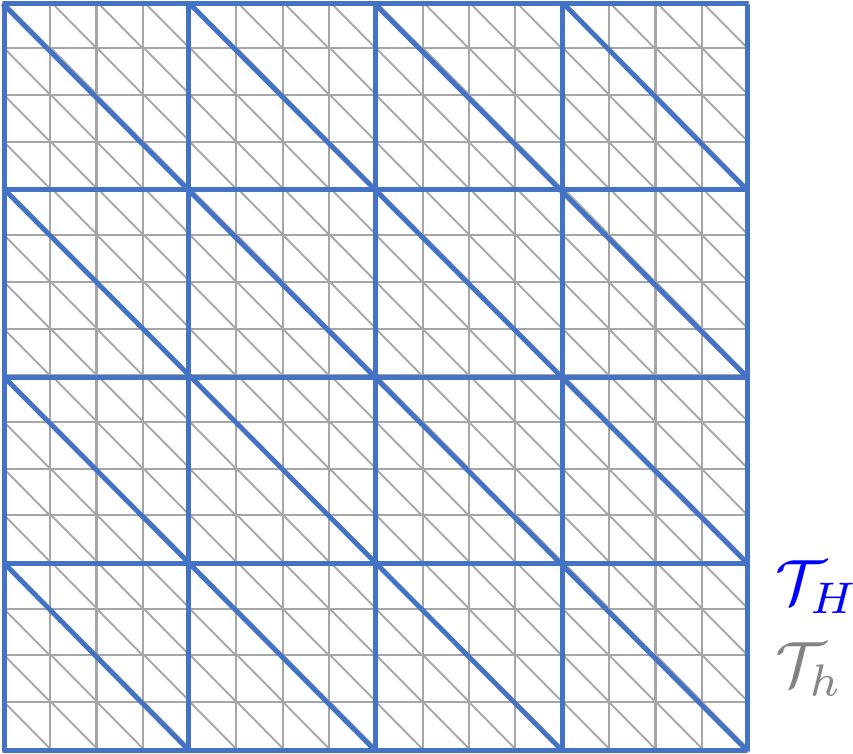}} \quad
  \subfloat[][]{\includegraphics[height=0.28\hsize]{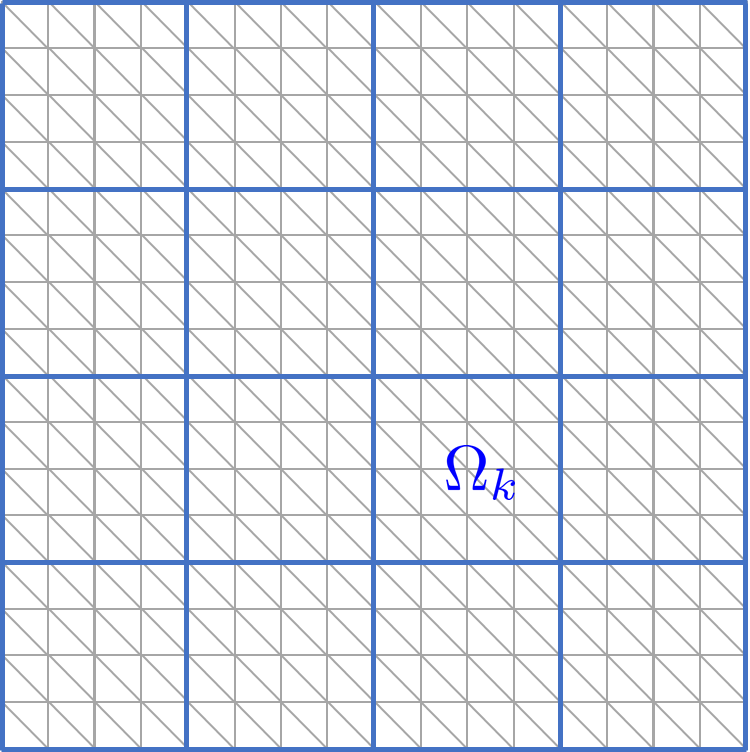}} \quad\quad
  \subfloat[][]{\includegraphics[height=0.28\hsize]{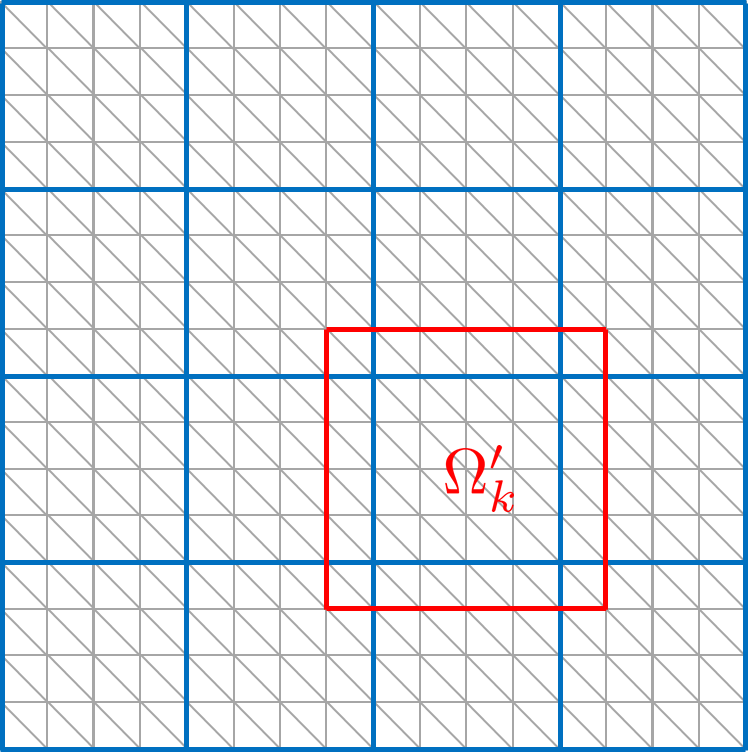}}
  \caption{Discretization and domain decomposition settings when $h = 1/2^4$, $H = 1/2^2$, and $\delta = h$.
  {\rm\bf (a)} Coarse triangulation $\cT_H$ and fine triangulation $\cT_h$.
  {\rm\bf (b)} Nonoverlapping domain decomposition $\{ \Omega_k \}_{k=1}^N$.
  {\rm\bf (c)} Overlapping domain decomposition $\{ \Omega_k ' \}_{k=1}^N$.}
  \label{Fig:DD}
\end{figure}

In the model $p$-Laplacian problem~\eqref{pLap_strong}, we set $p \in \{ 1.05, 1.1, 1.5, 5, 10, 20 \}$, $\Omega = [0,1]^2 \subset \mathbb{R}^2$, and $f = 1$.
The domain $\Omega$ is partitioned into $2 \times 1/H \times 1/H$ uniform triangles to form a coarse triangulation $\cT_H$ of $\Omega$.
We further refine $\cT_H$ to obtain a fine triangulation $\cT_h$, which consists of total $2 \times 1/h \times 1/h$ uniform triangles.
Each subdomain $\Omega_k$, $1 \leq k \leq N$~($N = 1/H \times 1/H$), is defined by a rectangular region consisting of two coarse triangles sharing a diagonal edge.
Then we extend $\Omega_k$ by adding its surrounding layers of fine triangles in $\cT_h$ with the width $\delta$ to construct $\Omega_k'$, so that $\{ \Omega_k' \}_{k=1}^N$ becomes an overlapping domain decomposition for $\Omega$.
If $\delta \in (0, H/2)$, then $\{ \Omega_k' \}_{k=1}^N$ can be colored with 4 colors in the way described in \cref{Lem:convex}.
The discretization and domain decomposition settings described above are illustrated in \cref{Fig:DD}.

In \cref{Alg:ASM}, we set $u^{(0)} = 0$ and $\tau = \tau_0 = 1/5$.
Local problems defined on $V_k$, $1 \leq k \leq N$, and coarse problems defined on $V_0$ are solved by the adaptive Newton method proposed in~\cite[Algorithm~2.1]{Mishchenko:2023}.
We use the stop criterion
\begin{equation*}
\left| \frac{F_k (w_k^{(n+1)}) - F_k (w_k^{(n)})}{F_k (w_k^{(n+1)})} \right| < 10^{-12}
\end{equation*}
for both local and coarse problems, where $F_k$ represents the energy functional corresponding to the local or coarse problems on $V_k$.

\begin{remark}
\label{Rem:First}
As alternatives to the adaptive Newton method used in this paper, which is a second-order optimization algorithm, first-order optimization algorithms~\cite{Teboulle:2018} can be adopted to solve the local and coarse problems.
These algorithms are generally easier to implement as they do not require the Hessian information of the energy functional but known to converge slower than second-order algorithms.
To accelerate the convergence rate of a first-order algorithm, several techniques such as the FISTA momentum~\cite{BT:2009}, restart scheme~\cite{OC:2015}, and backtracking~\cite{SGB:2014} can be employed.
\end{remark}

\begin{figure}
  \subfloat[][$p = 1.05$]{\includegraphics[width=0.31\hsize]{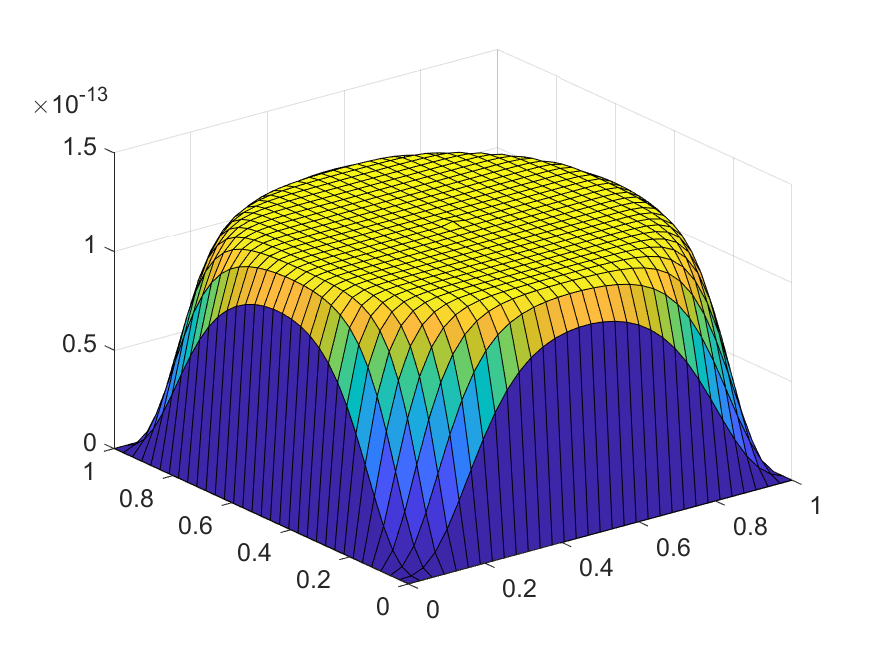}}
  \subfloat[][$p = 1.1$]{\includegraphics[width=0.31\hsize]{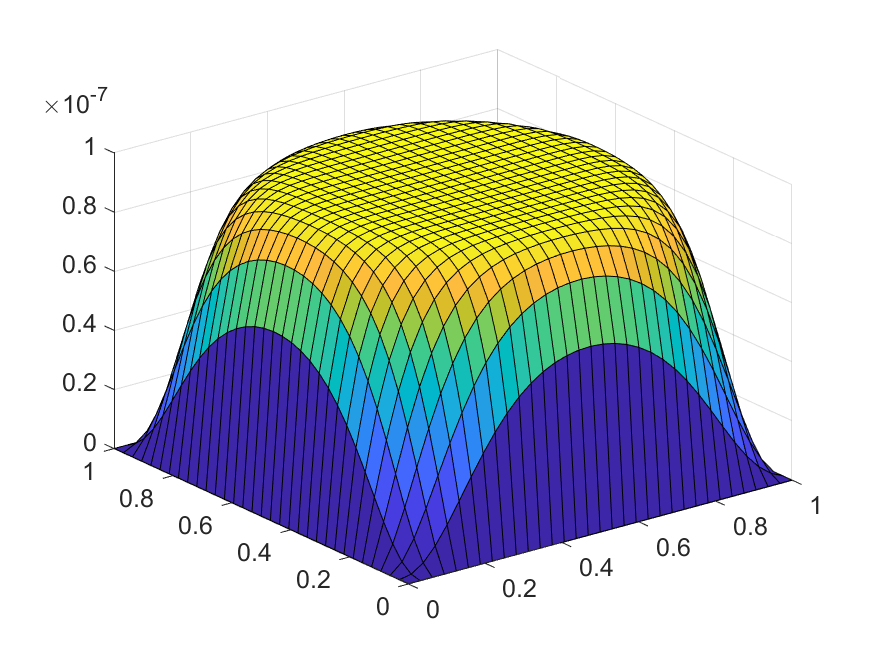}}
  \subfloat[][$p = 1.5$]{\includegraphics[width=0.31\hsize]{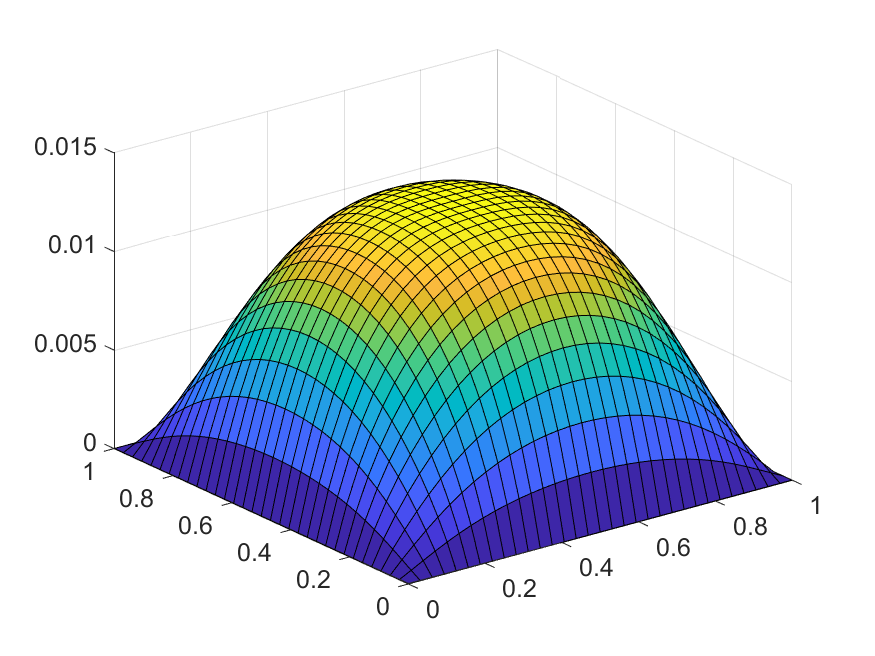}} \\
  
  \subfloat[][$p = 5$]{\includegraphics[width=0.31\hsize]{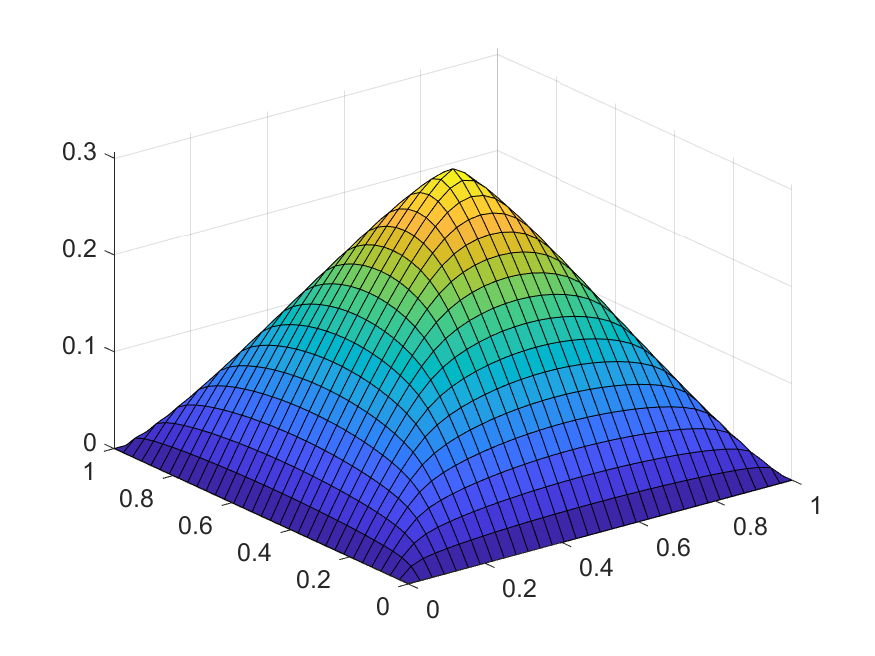}}
  \subfloat[][$p = 10$]{\includegraphics[width=0.31\hsize]{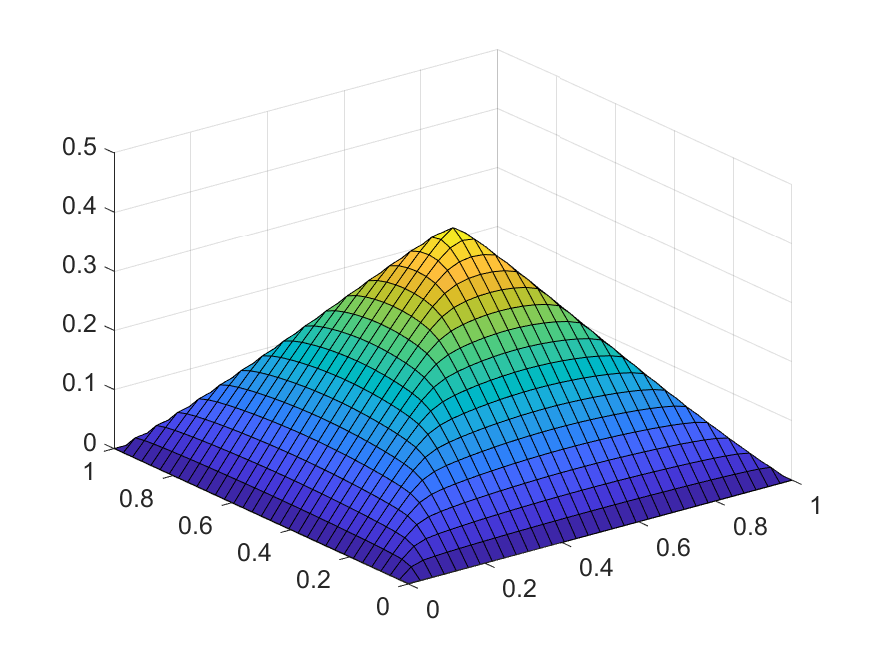}}
  \subfloat[][$p = 20$]{\includegraphics[width=0.31\hsize]{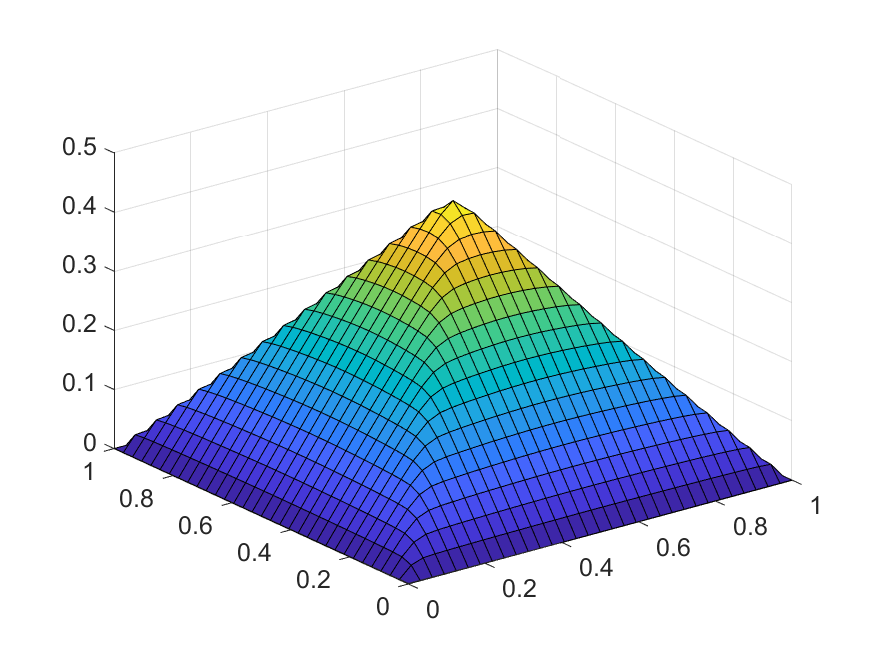}}
  \caption{Reference solutions of the $p$-Laplacian problem~\eqref{pLap_FEM}~($p \in \{ 1.05, 1.1, 1.5, 5, 10, 20 \}$) computed by the adaptive Newton method~\cite{Mishchenko:2023}~($h = 2^{-5}$).}
  \label{Fig:sol}
\end{figure}

A reference solution $u^* \in V$ for each $p$ and $h$ is computed by sufficiently many iterations of the adaptive Newton method applied to the full-dimension problem~\eqref{pLap_FEM}.
The computed reference solutions for $p \in \{ 1.05, 1.1, 1.5, 5, 10, 20 \}$ are plotted in \cref{Fig:sol}.
One can observe that for cases where $p$ is close to $1$, the reference solutions exhibit flat regions where the gradient vanishes.
This observation implies that when $p$ is close to $1$, the assumption in \cref{Thm:linear} that $\nabla u^*$ does not vanish may not hold.
On the other hand, for cases where $p$ is large, the reference solutions display peaks, leading to singular behavior in the solution.

\begin{figure}
  \subfloat[][$\delta = 2^0h$]{\includegraphics[width=0.31\hsize]{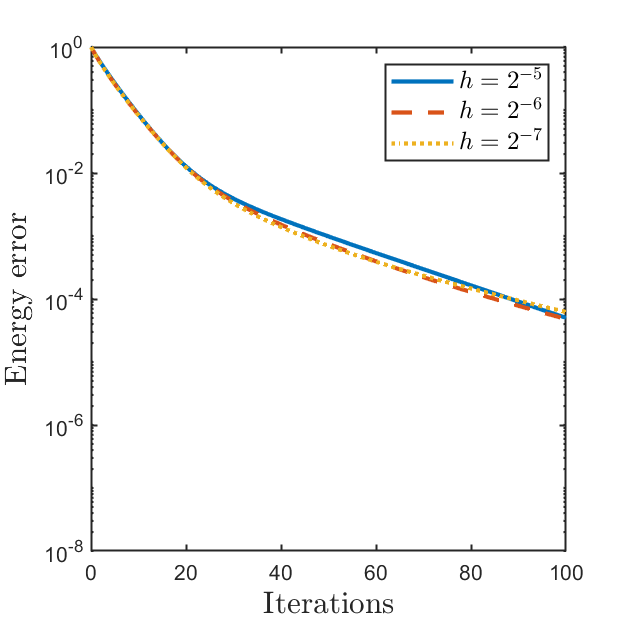}}
  \subfloat[][$\delta = 2^1h$]{\includegraphics[width=0.31\hsize]{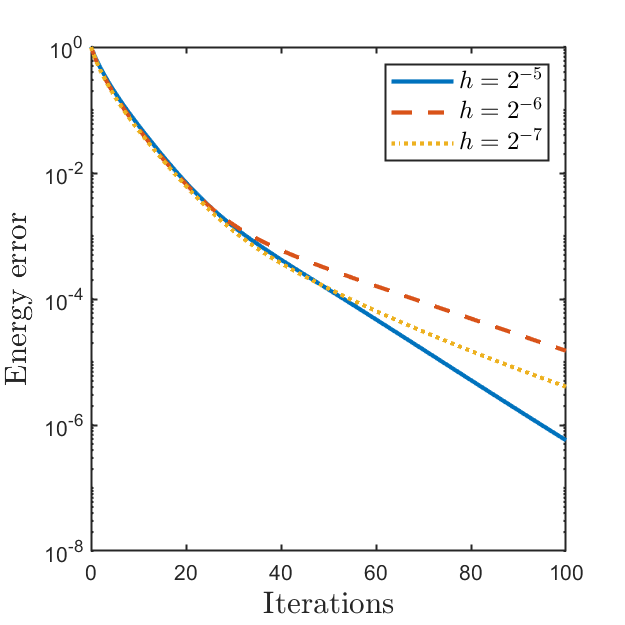}}
  \subfloat[][$\delta = 2^2h$]{\includegraphics[width=0.31\hsize]{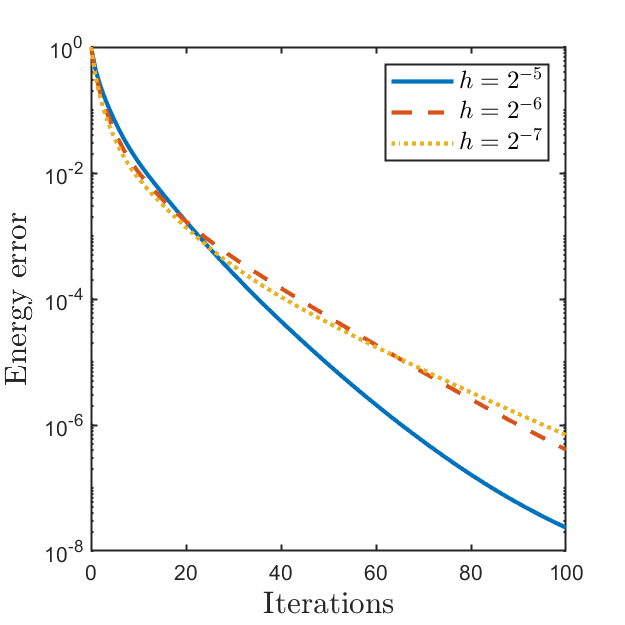}}
  \caption{Decay of the relative energy error~\eqref{rel_error} in the two-level additive Schwarz method~(\cref{Alg:ASM}) for the $p$-Laplacian problem~\eqref{pLap_FEM}~($p = 1.05$).
  Parameters $h$, $H$, and $\delta$ stand for the characteristic element size, subdomain size, and overlapping width among subdomains, respectively~($H/h = 2^3$).}
  \label{Fig:1.05}
\end{figure}

\begin{figure}
  \subfloat[][$\delta = 2^0h$]{\includegraphics[width=0.31\hsize]{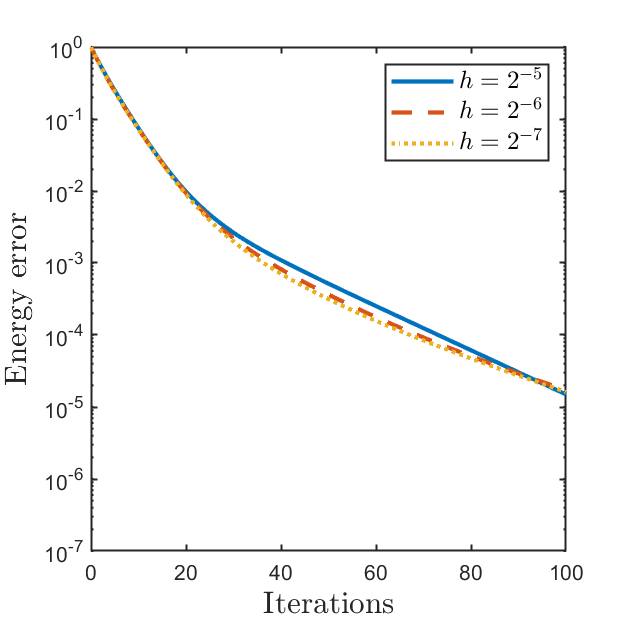}}
  \subfloat[][$\delta = 2^1h$]{\includegraphics[width=0.31\hsize]{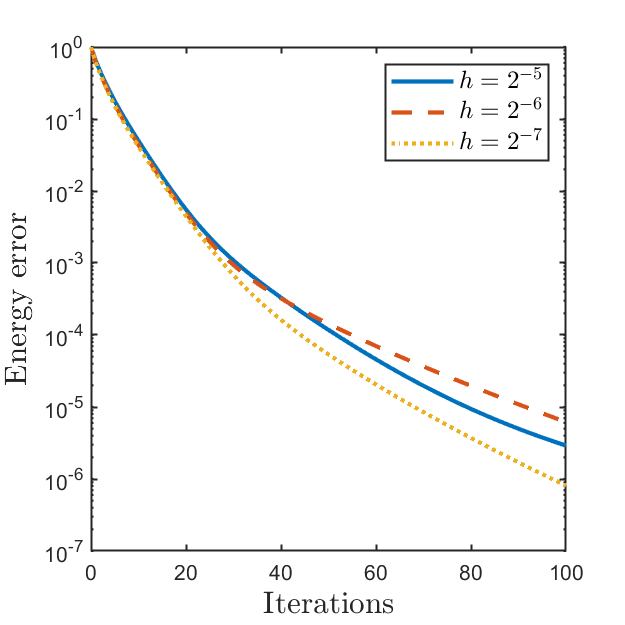}}
  \subfloat[][$\delta = 2^2h$]{\includegraphics[width=0.31\hsize]{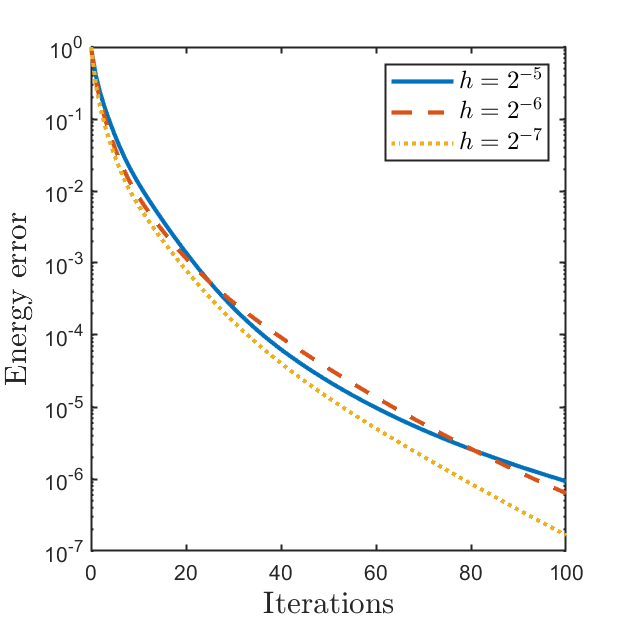}}
  \caption{Decay of the relative energy error~\eqref{rel_error} in the two-level additive Schwarz method~(\cref{Alg:ASM}) for the $p$-Laplacian problem~\eqref{pLap_FEM}~($p = 1.1$).
  Parameters $h$, $H$, and $\delta$ stand for the characteristic element size, subdomain size, and overlapping width among subdomains, respectively~($H/h = 2^3$).}
  \label{Fig:1.10}
\end{figure}

\begin{figure}
  \subfloat[][$\delta = 2^0h$]{\includegraphics[width=0.31\hsize]{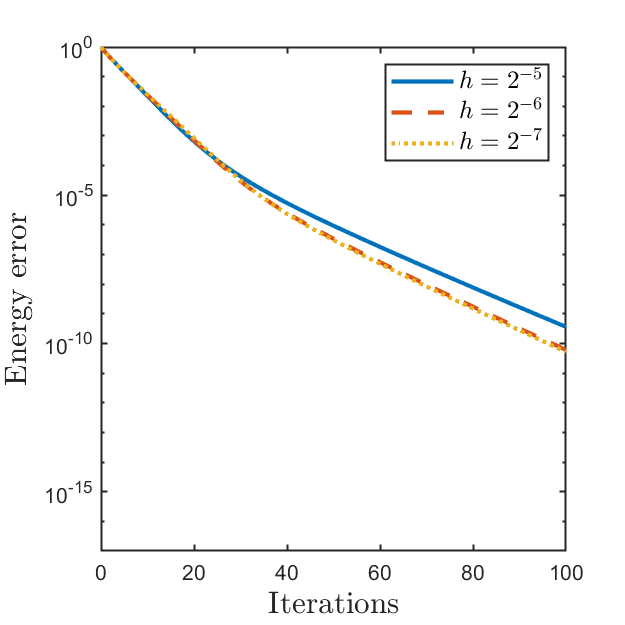}}
  \subfloat[][$\delta = 2^1h$]{\includegraphics[width=0.31\hsize]{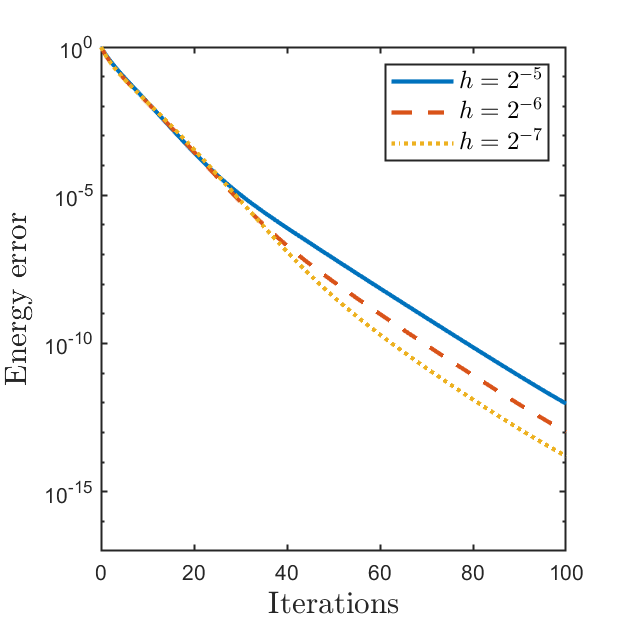}}
  \subfloat[][$\delta = 2^2h$]{\includegraphics[width=0.31\hsize]{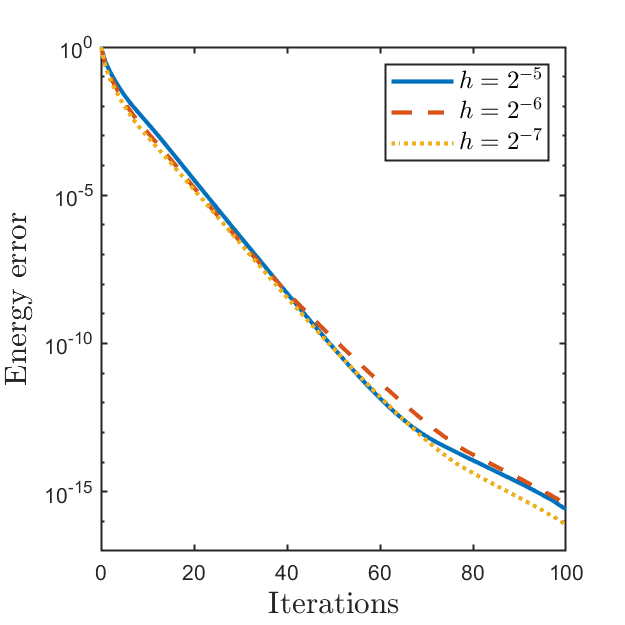}}
  \caption{Decay of the relative energy error~\eqref{rel_error} in the two-level additive Schwarz method~(\cref{Alg:ASM}) for the $p$-Laplacian problem~\eqref{pLap_FEM}~($p = 1.5$).
  Parameters $h$, $H$, and $\delta$ stand for the characteristic element size, subdomain size, and overlapping width among subdomains, respectively~($H/h = 2^3$).}
  \label{Fig:1.50}
\end{figure}

\begin{figure}
  \subfloat[][$\delta = 2^0h$]{\includegraphics[width=0.31\hsize]{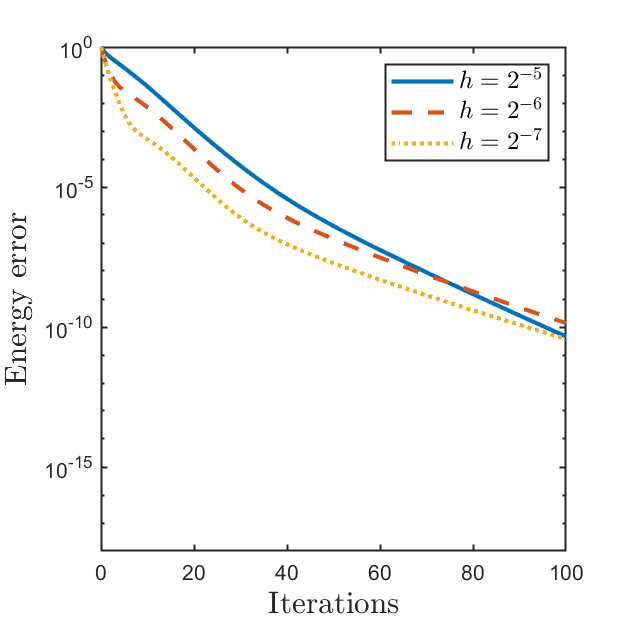}}
  \subfloat[][$\delta = 2^1h$]{\includegraphics[width=0.31\hsize]{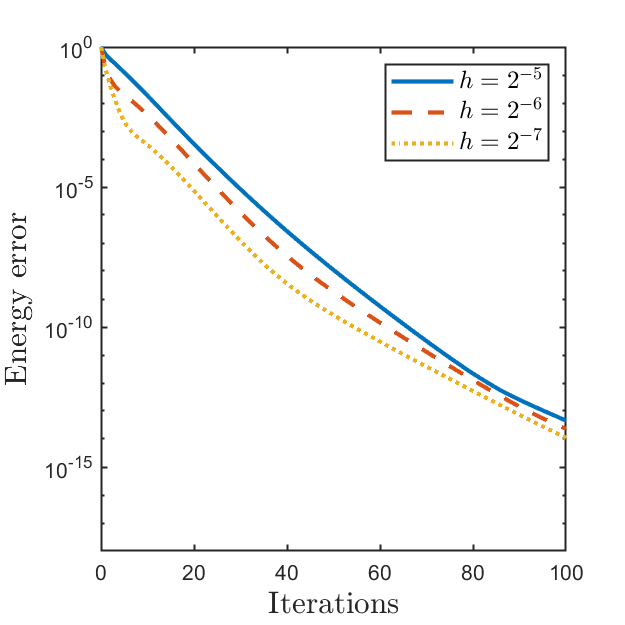}}
  \subfloat[][$\delta = 2^2h$]{\includegraphics[width=0.31\hsize]{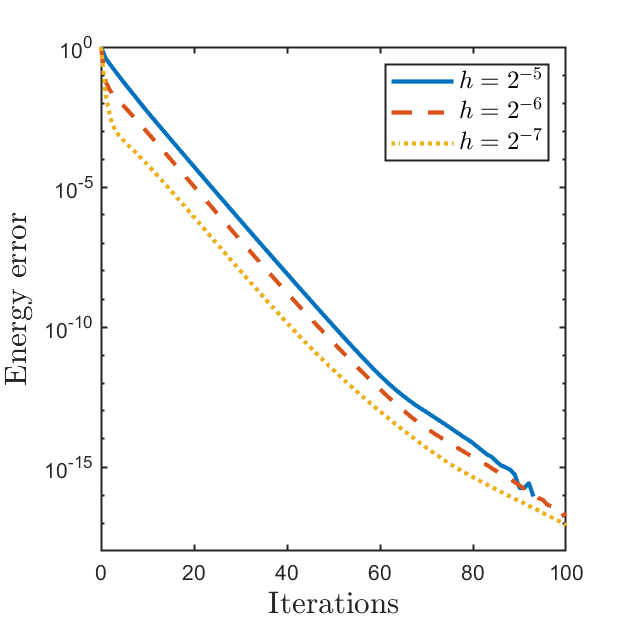}}
  \caption{Decay of the relative energy error~\eqref{rel_error} in the two-level additive Schwarz method~(\cref{Alg:ASM}) for the $p$-Laplacian problem~\eqref{pLap_FEM}~($p = 5$).
  Parameters $h$, $H$, and $\delta$ stand for the characteristic element size, subdomain size, and overlapping width among subdomains, respectively~($H/h = 2^3$).}
  \label{Fig:5.00}
\end{figure}

\begin{figure}
  \subfloat[][$\delta = 2^0h$]{\includegraphics[width=0.31\hsize]{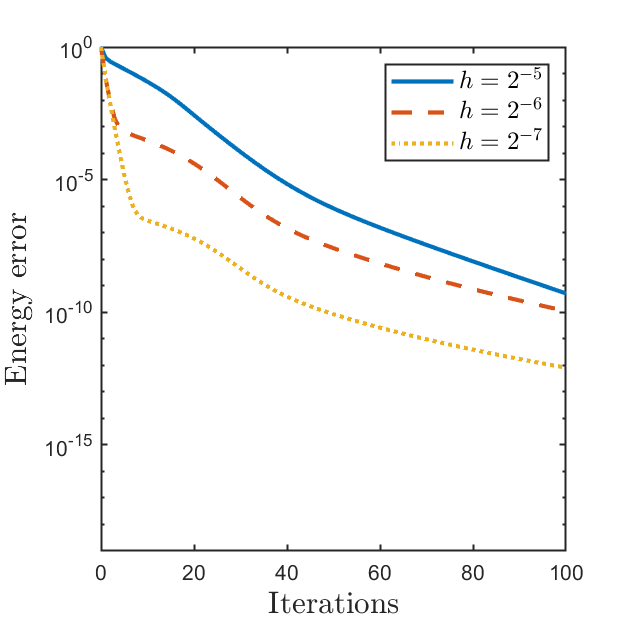}}
  \subfloat[][$\delta = 2^1h$]{\includegraphics[width=0.31\hsize]{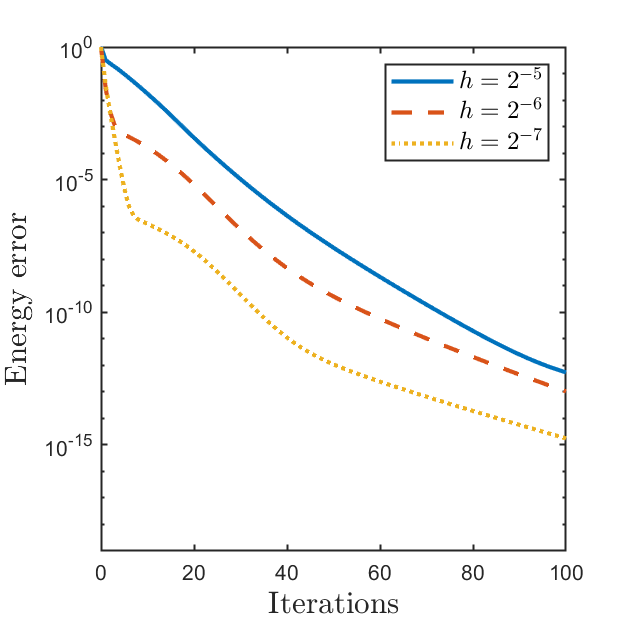}}
  \subfloat[][$\delta = 2^2h$]{\includegraphics[width=0.31\hsize]{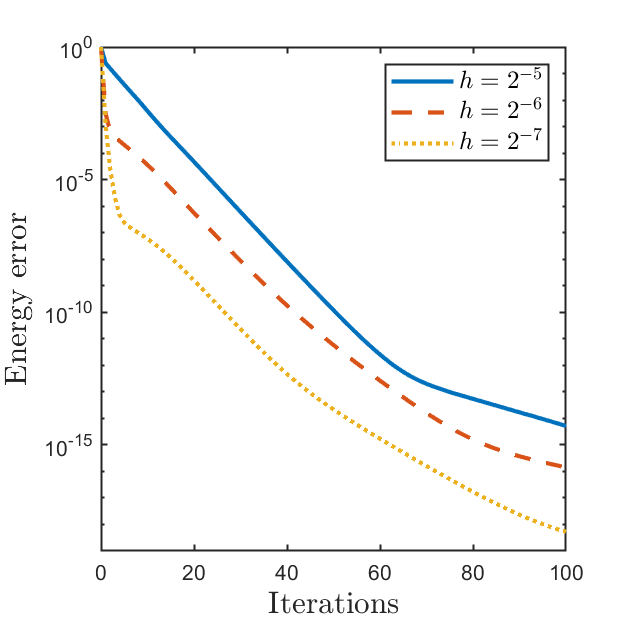}}
  \caption{Decay of the relative energy error~\eqref{rel_error} in the two-level additive Schwarz method~(\cref{Alg:ASM}) for the $p$-Laplacian problem~\eqref{pLap_FEM}~($p = 10$).
  Parameters $h$, $H$, and $\delta$ stand for the characteristic element size, subdomain size, and overlapping width among subdomains, respectively~($H/h = 2^3$).}
  \label{Fig:10.00}
\end{figure}

\begin{figure}
  \subfloat[][$\delta = 2^0h$]{\includegraphics[width=0.31\hsize]{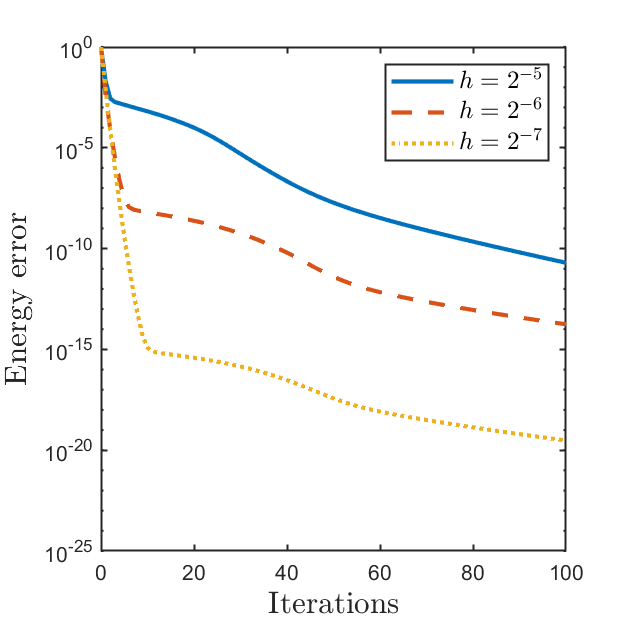}}
  \subfloat[][$\delta = 2^1h$]{\includegraphics[width=0.31\hsize]{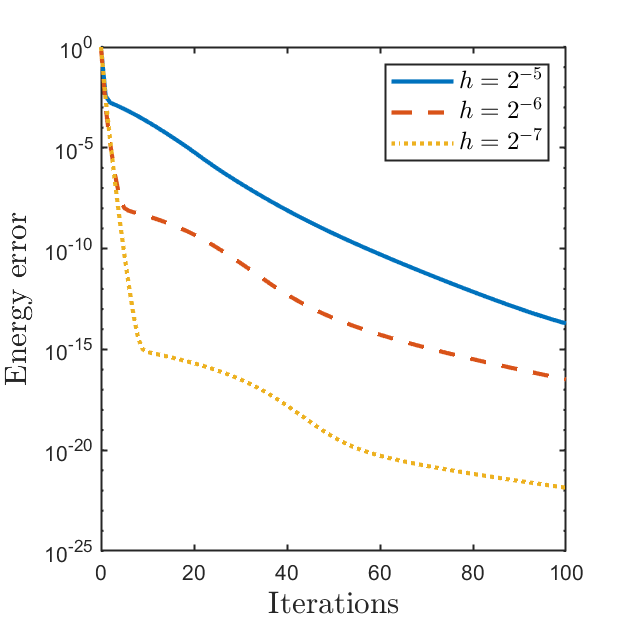}}
  \subfloat[][$\delta = 2^2h$]{\includegraphics[width=0.31\hsize]{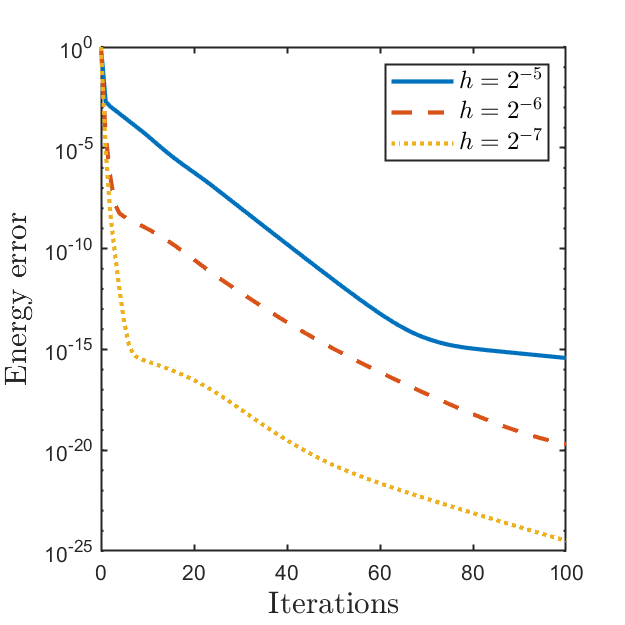}}
  \caption{Decay of the relative energy error~\eqref{rel_error} in the two-level additive Schwarz method~(\cref{Alg:ASM}) for the $p$-Laplacian problem~\eqref{pLap_FEM}~($p = 20$).
  Parameters $h$, $H$, and $\delta$ stand for the characteristic element size, subdomain size, and overlapping width among subdomains, respectively~($H/h = 2^3$).}
  \label{Fig:20.00}
\end{figure}

In \cref{Fig:1.05,Fig:1.10,Fig:1.50,Fig:5.00,Fig:10.00,Fig:20.00}, we depict the relative energy errors
\begin{equation}
\label{rel_error}
\frac{F(u^{(n)}) - F(u^*)}{ F(u^{(0)}) - F(u^*)}
\end{equation}
of \cref{Alg:ASM} under various settings on $p$, $h$, $H$, and $\delta$.
More precisely, in \cref{Fig:1.50,Fig:5.00}, we choose $p$ as moderate values $1.5$ and $5$, while in \cref{Fig:1.05,Fig:1.10}, $p$ is chosen very close to $1$~($p = 1.05, 1.1$), and in \cref{Fig:10.00,Fig:20.00}, $p$ is chosen large~($p = 10, 20$).
In all figures, $h$ and $H$ vary such that $H/h = 2^3$, and $\delta$ is chosen as $\delta \in \{ 2^0h, 2^1h, 2^2 h \}$.

In every case, we observe that the convergence curve of the relative energy error with respect to the number of iterations $n$ appears linear in the $x$-linear $y$-log scale plot when $n$ is large enough, consistent with our theoretical result presented in \cref{Thm:linear}.
It is noteworthy that even in cases where $p$ is very close to $1$~(see \cref{Fig:1.05,Fig:1.10}), where \cref{Thm:linear} cannot be applied due to the flat region in the solution $u^*$ as shown in \cref{Fig:sol}(a, b), the convergence curve still appears linear.
However, a theoretical explanation for the linear convergence in these cases is currently lacking.

On the other hand, for each $p$, we observe that the asymptotic convergence rate of \cref{Alg:ASM} shown in \cref{Fig:1.05,Fig:1.10,Fig:1.50,Fig:5.00,Fig:10.00,Fig:20.00} remains bounded when $h$ decreases keeping $H / \delta$ constant.
This behavior aligns with the dependence of $\gamma$ to $H/ \delta$ explained in \cref{Thm:linear}.
Moreover, this observation implies that \cref{Alg:ASM} is numerically scalable; the asymptotic linear convergence rate is uniformly bounded when the ratio of the subdomain size to the overlapping width is fixed.

\section{Conclusion}
\label{Sec:Conclusion}
In this paper, we developed a new convergence theory for additive Schwarz methods for boundary value problems involving the $p$-Laplacian.
To the best of our knowledge, our theory is the first theoretical result that explains the asymptotic linear convergence of additive Schwarz methods for the $p$-Laplacian.
Our work successfully bridges the gap between theory and practice by demonstrating that our theoretical findings align well with numerical results.

While the convergence theory of subspace correction methods for linear problems appears to be well-developed~\cite{LWXZ:2008,XZ:2002}, there remains a need for further research on the theory of subspace correction methods for nonlinear problems.
We believe that our result can serve as a foundation for the sharp convergence theory of general subspace correction methods for complex nonlinear problems.

\bibliographystyle{siamplain}
\bibliography{refs_Schwarz_pLap}

\end{document}